\documentclass[11pt,reqno]{amsart}

\usepackage{subfiles}

\usepackage{amssymb,latexsym}
\usepackage[margin=1in,headsep=.2in]{geometry}
\usepackage{graphicx}
\usepackage{mathtools}
\usepackage{color}
\usepackage[T1]{fontenc}
\usepackage{verbatim}
\usepackage{xurl}    
\usepackage{hyperref}
\hypersetup{pdfstartview=FitH, pdftitle=The Generalized Bergman Game}
\usepackage{cleveref}
\usepackage{amsmath}
\usepackage{mathdots}
\usepackage{breqn}
\usepackage{enumitem}
\usepackage{breqn}
\newcommand{\bburl}[1]{\textcolor{blue}{\url{#1}}}

\calclayout

\makeatletter
\newcommand{\monthyear}[1]{%
  \def\@monthyear{\uppercase{#1}}}
\newcommand{\volnumber}[1]{%
  \def\@volnumber{\uppercase{#1}}}
\newcommand{\journal}[1]{
    \def\@journal{\uppercase{#1}}
}
\AtBeginDocument{%
\def\ps@plain{\ps@empty
  \def\@oddfoot{\@monthyear \hfil \thepage}%
  \def\@evenfoot{\thepage \hfil \@volnumber}}
\def\ps@firstpage{\ps@plain}
\def\ps@headings{\ps@empty
  \def\@evenhead{%
    \setTrue{runhead}%
    \def\thanks{\protect\thanks@warning}%
    \@journal\hfil}%
  \def\@oddhead{%
    \setTrue{runhead}%
    \def\thanks{\protect\thanks@warning}%
    \hfill\uppercase{The Generalized Bergman Game}}%
  \let\@mkboth\markboth
  \def\@evenfoot{%
    \thepage \hfil \@volnumber}%
  \def\@oddfoot{%
    \@monthyear \hfil \thepage}%
  }%
\footskip=25pt
\pagestyle{headings}%
}
\makeatother

\theoremstyle{plain} 
\numberwithin{equation}{section}

\newtheorem{thm}{Theorem}[section]
\newtheorem{theorem}[thm]{Theorem}

\newtheorem{conjecture}[thm]{Conjecture}

\newtheorem{prop}[thm]{Proposition}
\newtheorem{porism}[thm]{Porism}
\newtheorem{lemma}[thm]{Lemma}

\newtheorem{cor}[thm]{Corollary}

\newtheorem*{notate}{Notation}

\newtheorem*{rmk}{Remark}

\theoremstyle{definition}
\newtheorem{defn}[thm]{Definition}
\crefalias{defn}{definition}
\newtheorem{example}[thm]{Example}

\usepackage{amsthm, amsfonts, times}
\usepackage{epstopdf}

\usepackage[maxnames=10,minnames=10,style=alphabetic]{biblatex}
\usepackage{mathrsfs}
\usepackage{centernot}
\usepackage{xfrac}
\usepackage{eufrak}
\usepackage{stmaryrd}
\usepackage{bm}
\usepackage[all,cmtip]{xy}
\usepackage{float}
\usepackage[ruled,vlined]{algorithm2e}
\usepackage{tikz-cd}
\usetikzlibrary{positioning,arrows}
\usetikzlibrary{calc}
\usetikzlibrary{matrix}
\usepackage[colorinlistoftodos]{todonotes}

\newcommand{\notdone}[1]{\todo[inline, color=red!40]{#1}}

\makeatletter
\providecommand\@dotsep{5}
\def\listtodoname{TODOS: Not Done (Red), Cite (Orange), Blue (Notation)}
\def\listoftodos{\@starttoc{tdo}\listtodoname}
\makeatother

\newcommand{\val}{v}
\newcommand{\valt}{\widetilde{v}}
\newcommand{\valtabs}{{\widetilde{v}_+}}
\newcommand{\chips}[1]{\abs{#1}}
\newcommand{\ind}{\mathcal{I}nd}
\newcommand{\len}{\ell}
\newcommand{\gap}{g}

\newcommand{\NIPLRS}{non-increasing PLRS\xspace}

\newcommand{\PLRSs}{PLRSs\xspace}
\newcommand{\PLRS}{PLRS\xspace}

\newcommand{\slcr}{\mathcal{S}_{\mathscr{L}}\mathcal{C}_{\mathscr{R}}}
\newcommand{\slcl}{\mathcal{S}_{\mathscr{L}}\mathcal{C}_{\mathscr{L}}}
\newcommand{\srcr}{\mathcal{S}_{\mathscr{R}}\mathcal{C}_{\mathscr{R}}}
\newcommand{\srcl}{\mathcal{S}_{\mathscr{R}}\mathcal{C}_{\mathscr{L}}}

\newcommand{\crsl}{\mathcal{C}_{\mathscr{R}}\mathcal{S}_{\mathscr{L}}}
\newcommand{\crsr}{\mathcal{C}_{\mathscr{R}}\mathcal{S}_{\mathscr{R}}}
\newcommand{\clsl}{\mathcal{C}_{\mathscr{L}}\mathcal{S}_{\mathscr{L}}}
\newcommand{\clsr}{\mathcal{C}_{\mathscr{L}}\mathcal{S}_{\mathscr{R}}}

\newcommand{\floor}[1]{\left\lfloor#1\right\rfloor}
\newcommand{\ceil}[1]{\left\lceil#1\right\rceil}
\newcommand{\abs}[1]{\left|#1\right|}

\newcommand{\N}{\mathbb N}
\newcommand{\Q}{\mathbb Q}
\newcommand{\R}{\mathbb R}
\newcommand{\Z}{\mathbb Z}

\def\dooutline{0}

\monthyear{August 2021}
\volnumber{{}}
\journal{{}}

\thanks{This work was supported by NSF Grants DMS1561945 and DMS1659037. This work was also supported by Williams College and the University of Michigan. We thank the participants of the 2021 Williams SMALL REU for many helpful conversations and the referee for constructive comments.}

\title{The Generalized Bergman Game}

\author{Benjamin Baily}
\email{\textcolor{blue}{\href{mailto:bmb2@williams.edu}{bmb2@williams.edu}}}
\address{Department of Mathematics and Statistics, Williams College, Williamstown MA 01267}

\author{Justine Dell}
\email{\textcolor{blue}{\href{mailto:jdell@haverford.edu}{jdell@haverford.edu}}}
\address{Department of Mathematics and Statistics, Haverford College, Haverford, PA 19041}

\author{Irfan Durmi\'{c}}
\email{\textcolor{blue}{\href{mailto:id5@williams.edu}{id5@williams.edu}}}
\address{Department of Mathematics and Statistics, Williams College, Williamstown MA 01267}

\author{Henry L. Fleischmann}
\email{\textcolor{blue}{\href{mailto:henryfl@umich.edu}{henryfl@umich.edu}}}
\address{Department of Mathematics, University of Michigan, Ann Arbor, MI 48109}

\author{Faye Jackson}
\email{\textcolor{blue}{\href{mailto:alephnil@umich.edu}{alephnil@umich.edu}}}
\address{Department of Mathematics, University of Michigan, Ann Arbor, MI 48109}

\author{Isaac Mijares}
\email{\textcolor{blue}{\href{mailto:rim1@williams.edu}{rim1@williams.edu}}}
\address{Department of Mathematics and Statistics, Williams College, Williamstown, MA 01267}

\author{Steven J. Miller}
\email{\textcolor{blue}{\href{mailto:sjm1@williams.edu}{sjm1@williams.edu}},  \textcolor{blue}{\href{Steven.Miller.MC.96@aya.yale.edu}{Steven.Miller.MC.96@aya.yale.edu}}}
\address{Department of Mathematics and Statistics, Williams College, Williamstown, MA 01267}

\author{Ethan Pesikoff}
\email{\textcolor{blue}{\href{mailto:ethan.pesikoff@yale.edu}{ethan.pesikoff@yale.edu}}}
\address{Department of Mathematics,
Yale University,
New Haven, CT 06520}

\author{Luke Reifenberg}
\email{\textcolor{blue}{\href{mailto:lreifenb@nd.edu}{lreifenb@nd.edu}}}
\address{Department of Mathematics, University of Notre Dame, Notre Dame, IN 46556}

\author{Alicia Smith Reina}
\email{\textcolor{blue}{\href{mailto:ags6@williams.edu}{ags6@williams.edu}}}
\address{Department of Mathematics and Statistics, Williams College, Williamstown, MA 01267}

\author{Yingzi Yang}
\email{\textcolor{blue}{\href{mailto:yyingzi@umich.edu}{yyingzi@umich.edu}}}
\address{Department of Mathematics, University of Michigan, Ann Arbor, MI 48109}

\date{\today}


\begin{document}


\begin{abstract}
    Every positive integer may be written uniquely as a base-$\beta$ decomposition--that is a legal sum of powers of $\beta$--where $\beta$ is the dominating root of a non-increasing positive linear recurrence sequence. Guided by earlier work on a two-player game which produces the Zeckendorf Decomposition of an integer (see \cite{baird-smithGeneralizedZeckendorfGame2019}), we define a broad class of two-player games played on an infinite tuple of non-negative integers which decompose a positive integer into its base-$\beta$ expansion. We call this game the Generalized Bergman Game. We prove that the longest possible Generalized Bergman game on an initial state $S$ with $n$ summands terminates in $\Theta(n^2)$ time, and we also prove that the shortest possible Generalized Bergman game on an initial state terminates between $\Omega(n)$ and $O(n^2)$ time. We also show a linear bound on the maximum length of the tuple used throughout the game.
\end{abstract}

\maketitle
\tableofcontents

\ifthenelse{\dooutline > 0}{
    \listoftodos
}{
\presetkeys{todonotes}{disable}{}
}

\notdone{Check Paper-y Things, such as Thanks, SubjClass, author emails etc}
\notdone{Fix Formatting}

\section{Introduction}\label{sec:intro}

\subsection{History and Motivation}\label{subs:history-motivation}

Every integer $n > 0$ can be written uniquely as a sum of non-adjacent Fibonacci numbers $\{F_n: F_1=1, F_2=2, F_{n+1}=F_n+F_{n-1}\}$, called its Zeckendorf decomposition. For example
\[
    2021 = 1597 + 377 + 34 + 13 = F_{15} + F_{13} + F_8 + F_6.
\]
Previous literature has extensively analyzed generalizations of this theorem to other recurrences using number theoretic and probabilistic techniques (see \cite{hoggattGeneralizedZeckendorfTheorem1972,kellerGeneralizationsZeckendorfTheorem,kologluNumberSummandsZeckendorf2010} and \cite{millerFibonacciNumbersCentral2012,kologluNumberSummandsZeckendorf2010,gapsCentralLimitTheorems2016,gapsSummandsLatticeZeck,summandsGapsGeneralizedZeck} respectively). Earlier work (see \cite{baird-smithZeckendorfGame2020}) analyzed a combinatorial game played on an infinite tuple of Fibonacci numbers; starting with $n$ copies of $F_1$ players alternate by using moves arising from the Fibonacci recurrence which over time consolidate the many original copies of $F_1$ into a few instances of larger $F_i$.  The winner is the player who moves last and consequently forms the Zeckendorf decomposition of $n$. For $n > 2$, the second player to move has a winning strategy (though the proof is non-constructive), and all games take on the order of $n$ moves.

Our work takes the Zeckendorf Game and removes its boundary conditions (special moves only allowed at the left edge); this new game--the Bergman Game\footnote{Named for George Bergman, who discovered base-$\varphi$ decompositions in \cite{bergmanNumberSystemIrrational1957}.} --is now played on a doubly infinite tape. Playing the Bergman game produces the unique base-$\varphi$ decomposition of $n$, where $\varphi$ is the golden mean. For example:
\[
    2021 = \varphi^{-16} + \varphi^{-11} + \varphi^{-6} + \varphi^{-3} + \varphi + \varphi^5 + \varphi^{10} + \varphi^{13} + \varphi^{15}.
\]
Previous research has extensively studied base-$\beta$ representations for any real number $\beta > 1$ via Ergodic theory, symbolic dynamics, and algebraic number theory (see \cite{VExpansionsSymbolicDynamics1989,dajaniSymbolicDynamicsExpansions, grabnerGeneralizedZeckendorfExpansions1994,frougnyFiniteBetaexpansions1992}).

We further define a Generalized Bergman Game played on any $\beta > 1$ which satisfies the characteristic polynomial of a Non-Increasing Positive Linear Recurrence Sequence (\NIPLRS). Our research thus presents a new way of studying these base-$\beta$ representations using classes of games similar to the Zeckendorf Game. For example, these games provide an elementary proof that any number which is a finite sum of powers of $\beta$ with non-negative coefficients has a finite base-$\beta$ decomposition (see \Cref{cor:finite-beta-decomposition}). Although we do not prove the complete known results in full generality, the corresponding literature uses ergodic measure theory rather than elementary methods \cite{frougnyFiniteBetaexpansions1992}. Our research analyzes the termination time of the Generalized Bergman Game, showing tight bounds on the number of moves in such games.

\subsection{Basic Definitions}\label{subs:basic-defns}

\begin{defn}\label{defn:bergman-game}
    The Bergman Game is a two player turn-based game played on a doubly infinite tape. A game state of the Bergman Game, denoted by $S$, is of the form $S = (\dots,S(-1), S(0), S(1), \dots)$, where for all the indices $i \in \Z$, $S(i) \in \Z_{\geq 0}$ and there is only a finite number of $i$ such that $S(i) \neq 0$. We say that $S$ has $S(i)$ ``summands'' in index $i$.

    The Bergman Game begins with some starting game state, and then two players take turns making moves, choosing either to split or combine. The combine and split moves are defined as follows.
    \begin{itemize}
        \item\label{combine} Combine: If both $S(i-2)\geq 1$ and $S(i-1)\geq 1$, then decrease $S(i-2),S(i-1)$ each by $1$ and increase $S(i)$ by $1$. E.g.,
            \[
                (1, 1,0) \rightarrow (0,0,1).
            \]
        \item\label{split} Split: If there exists an $i$ such that $S(i)\geq 2$, then decrease $S(i)$ by $2$ and increase $S(i-2)$ by $1$ and increase $S(i+1)$ by $1$. E.g.,
            \[
                (0,0,2,0) \rightarrow (1,0,0,1).
            \]
    \end{itemize}
    The game is played until one player can neither split nor combine, and the last player to make a move wins.

    The entire sequence of play (the ordered collection of the game states which appeared over the course of play) is called a ``game'', (usually) denoted by $G$.
\end{defn}

\begin{notate}
    For convenience, we often abbreviate the game state
    \[
        (\dots,0,0,0,S(a),S(a+1),\dots,S(b),0,0,0,0,\dots)
    \] as the finite tuple
    \[
        {}_{a}(S(a),\dots,S(b)){}_{b}
    \]
    where $a$ is the leftmost non-zero valued index and $b$ is the rightmost non-zero valued index. We often also use the notation ${}_aS_b$ to refer to a game state $S$ with leftmost nonzero index $a$ and rightmost nonzero index $b$. This provides a convenient way to notate shifting a game state; for example ${}_0S_{b - a}$ refers to the game state where we have shifted ${}_aS_b$ to have leftmost nonzero entry at the zeroth index. If there is only one index with a non-zero number of summands, we refer to the game state as ${}_a(n)$.
\end{notate}

\begin{example}
    We provide an example of two possible Bergman Games which share a common first four moves and then diverge from each other at move five (\Cref{fig:non-determ-bergman-game}). This shows that the Bergman Game is non-deterministic on the initial game state ${}_0(n)$ in general, as the game may be played down either shown path, each of which gives a different winner.
\end{example}

\begin{figure}[H]
    \centering
    \scalebox{0.7}{\begin{tikzpicture}[cell/.style={rectangle,draw=black},
space/.style={minimum height=0.7em,matrix of nodes,row sep=-\pgflinewidth,column sep=-\pgflinewidth,column 1/.style={font=\ttfamily}}, scale=0.5, every node/.style={transform shape}]
        \matrix (first) [space, column 1/.style={nodes={cell,minimum width=2em}}]
        {
            6 \\
        };
        \matrix (second) [below=of first, space, column 1/.style={nodes={cell,minimum width=2em}}, column 2/.style={nodes={cell,minimum width=2em}}, column 3/.style={nodes={cell,minimum width=2em}}, column 4/.style={nodes={cell,minimum width=2em}}]
        {
            1 & 0 & 4 & 1 \\
        };
        \draw[->, red] (first) -- (second);
        
        \matrix (thirdc) [below=of second, space, column 1/.style={nodes={cell,minimum width=2em}}, column 2/.style={nodes={cell,minimum width=2em}}, column 3/.style={nodes={cell,minimum width=2em}}, column 4/.style={nodes={cell,minimum width=2em}}, column 5/.style={nodes={cell,minimum width=2em}}]
        {
            1 & 0 & 3 & 0 & 1 \\
        };
        
        \draw[->, blue] (second) -- (thirdc);
        
        \matrix (fourthcs) [below=of thirdc, space, column 1/.style={nodes={cell,minimum width=2em}}, column 2/.style={nodes={cell,minimum width=2em}}, column 3/.style={nodes={cell,minimum width=2em}}, column 4/.style={nodes={cell,minimum width=2em}}, column 5/.style={nodes={cell,minimum width=2em}}]
        {
            2 & 0 & 1 & 1 & 1 \\
        };
        
        \draw[->, red] (thirdc) -- (fourthcs);
        
        \matrix (fifthcss) [below=of fourthcs, space, column 1/.style={nodes={cell,minimum width=2em}}, column 2/.style={nodes={cell,minimum width=2em}}, column 3/.style={nodes={cell,minimum width=2em}}, column 4/.style={nodes={cell,minimum width=2em}}, column 5/.style={nodes={cell,minimum width=2em}}, column 6/.style={nodes={cell,minimum width=2em}}, column 7/.style={nodes={cell,minimum width=2em}}]
        {
            1 & 0 & 0 & 1 & 1 & 1 & 1 \\
        };
        
        \draw[->, red] (fourthcs) -- (fifthcss);
        
        \matrix (sixthcsscr) [below right=1cm and -3cm of fifthcss, space, column 1/.style={nodes={cell,minimum width=2em}}, column 2/.style={nodes={cell,minimum width=2em}}, column 3/.style={nodes={cell,minimum width=2em}}, column 4/.style={nodes={cell,minimum width=2em}}, column 5/.style={nodes={cell,minimum width=2em}}, column 6/.style={nodes={cell,minimum width=2em}}, column 7/.style={nodes={cell,minimum width=2em}}, column 7/.style={nodes={cell,minimum width=2em}},  column 8/.style={nodes={cell,minimum width=2em}}]
        {
            1 & 0 & 0 & 1 & 1 & 0 & 0 & 1\\
        };
        
        \matrix (sixthcsscl) [below left=1cm and -3cm of fifthcss, space, column 1/.style={nodes={cell,minimum width=2em}}, column 2/.style={nodes={cell,minimum width=2em}}, column 3/.style={nodes={cell,minimum width=2em}}, column 4/.style={nodes={cell,minimum width=2em}}, column 5/.style={nodes={cell,minimum width=2em}}, column 6/.style={nodes={cell,minimum width=2em}}, column 7/.style={nodes={cell,minimum width=2em}}]
        {
            1 & 0 & 0 & 1 & 0 & 0 & 2 \\
        };
        
        \draw[->, blue] (fifthcss) -- (sixthcsscr);
        \draw[->, blue] (fifthcss) -- (sixthcsscl);
                
        \matrix (seventhcsscrc) [below=of sixthcsscr, space, column 1/.style={nodes={cell,minimum width=2em}}, column 2/.style={nodes={cell,minimum width=2em}}, column 3/.style={nodes={cell,minimum width=2em}}, column 4/.style={nodes={cell,minimum width=2em}}, column 5/.style={nodes={cell,minimum width=2em}}, column 6/.style={nodes={cell,minimum width=2em}}, column 7/.style={nodes={cell,minimum width=2em}}, column 7/.style={nodes={cell,minimum width=2em}},  column 8/.style={nodes={cell,minimum width=2em}}]
        {
            1 & 0 & 0 & 0 & 0 & 1 & 0 & 1\\
        };
        
        \matrix (seventhcsscls) [below=of sixthcsscl, space, column 1/.style={nodes={cell,minimum width=2em}}, column 2/.style={nodes={cell,minimum width=2em}}, column 3/.style={nodes={cell,minimum width=2em}}, column 4/.style={nodes={cell,minimum width=2em}}, column 5/.style={nodes={cell,minimum width=2em}}, column 6/.style={nodes={cell,minimum width=2em}}, column 7/.style={nodes={cell,minimum width=2em}},  column 8/.style={nodes={cell,minimum width=2em}}]
        {
            1 & 0 & 0 & 1 & 1 & 0 & 0 & 1 \\
        };
        
        \draw[->, blue] (sixthcsscr) -- (seventhcsscrc);
        \draw[->, red] (sixthcsscl) -- (seventhcsscls);
        
        \matrix (eighthcssclsc) [below=of seventhcsscls, space, column 1/.style={nodes={cell,minimum width=2em}}, column 2/.style={nodes={cell,minimum width=2em}}, column 3/.style={nodes={cell,minimum width=2em}}, column 4/.style={nodes={cell,minimum width=2em}}, column 5/.style={nodes={cell,minimum width=2em}}, column 6/.style={nodes={cell,minimum width=2em}}, column 7/.style={nodes={cell,minimum width=2em}},  column 8/.style={nodes={cell,minimum width=2em}}]
        {
            1 & 0 & 0 & 0 & 0 & 1 & 0 & 1 \\
        };
        
        \draw[->, blue] (seventhcsscls) -- (eighthcssclsc);
    \end{tikzpicture}}
    \caption{A Non-Determistic Bergman Game with initial state ${}_0(6){}_{0}$, Player One wins in the left branch and Player Two wins in the right branch.  Red arrows mark splits, and blue combines.}
    \label{fig:non-determ-bergman-game}
\end{figure}
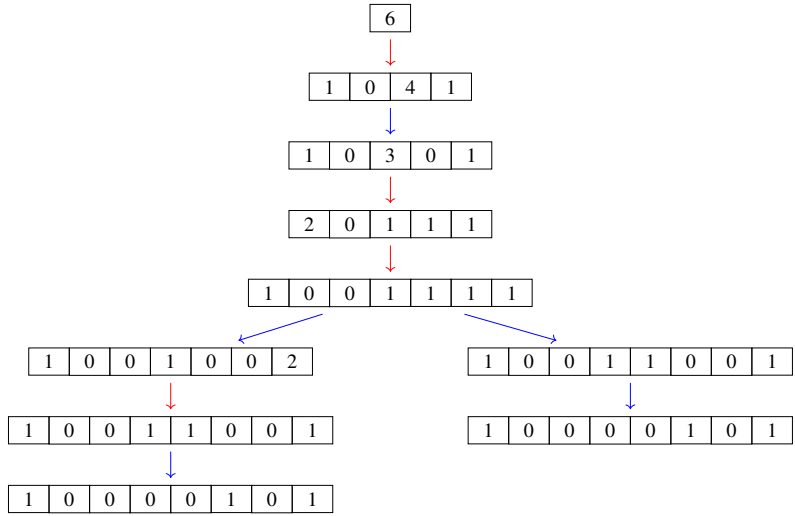

\begin{defn}\label{defn:chips}
    We define $\chips{S}$ as the number of summands in $S$. More formally, $\chips{S} \coloneqq \sum_j S(j)$. Note that the number of summands throughout a game is a non-increasing monovariant, as the number of summands remains the same with a split move and decreases by one with a combine.
\end{defn}

\begin{example}
    Let $S_i$ denote the initial game state from \Cref{fig:non-determ-bergman-game}, and $S_f$ the final state.  Then $\chips{S_i}=6$ and $\chips{S_f}=3$.
\end{example}

\begin{defn}\label{defn:plrs}
    We say a sequence is a Positive Linear Recurrence Sequence (PLRS) if it is given by a linear recurrence with characteristic polynomial $x^k-c_1x^{k - 1} - \cdots - c_k$ for some integers $c_i$ with $c_1,c_k > 0$ and $k\geq 2$.  We say it is non-increasing (\NIPLRS) if $c_1\geq c_2\geq\dots\geq c_k>0$. For convenience if $j > k$ we let $c_j = 0$.
\end{defn}

\begin{defn}\label{defn:generalized-bergman-game}
    The Generalized Bergman Game derives its moves from the Fibonacci recurrence relation. We can generalize it to an arbitrary \NIPLRS of depth at least two, which yields the Generalized Bergman Game (GBG). The Gneralized Bergman Game is also played on states which are doubly infinite tuples $S = (\ldots, S(-1), S(0), S(1), \ldots)$ of non-negative integers.

    On a given \NIPLRS with characteristic polynomial $x^k-c_1x^{k-1}-\dots-c_k$ the combine and split moves are given as follows.
    \begin{itemize}
        \item Combine: if there is an $i$ such that $S(i+j) \geq c_{k-j}$ for each $j \in \{0,1,2,\dots,k\}$ then decrease each $S(i+j)$ by $c_{k-j}$ and increase $S(i+k+1)$ by $1$. E.g.,
        \[
            (c_k,\dots,c_1,0) \to (\underbrace{0,...0}_{k},1).
        \]
        \item Split of type $1 \leq p \leq k - 1$: If there is an $i$ such that $S(i) \geq c_{p}+1$ and $S(i +j) \geq c_{p - j}$ for $j \in \{1, \ldots, p - 1\}$, decrease $S(i + j)$ by $c_{p - j}$, and then decrease $S(i)$ by $c_{p}+1$ and increase $S(i+p)$ by $1$ and increase $S(i-j)$ each by $d_{p,j}$ for $j \in \{1,2,3,\dots,k \}$, where $d_{p,j} \coloneqq c_j-c_{j+p}$. If $j +p > k$ then $d_{p, j} =  c_j$. E.g. a type one split has the form,
        \[
            (\underbrace{0,\dots,0}_{k},c_{1}+1,0) \rightarrow (d_{1,k},d_{1,k-1},\dots,d_{1,1},0,1).
        \]
        A split of type $p$ has the form
        \[
            (0, \ldots, 0, c_p + 1, c_{p - 1}, \ldots, c_1, 0) \rightarrow (d_{p, k}, d_{p, k - 1}, \ldots, d_{p, 1}, 0, 0, \ldots, 0, 1).
        \]
        Note that we can think of this as the combination of a ``reverse combine'' at index $i$, and then a combine performed into index $i + p$. This is illustrated below:
        \begin{align*}
            (0, \ldots, 0, c_p + 1, c_{p - 1}, \ldots, c_1, 0) &\rightarrow (c_k, \ldots, c_1, c_p, c_{p - 1}, \ldots, c_1, 0) \\
            &\rightarrow (d_{p, k}, d_{p, k - 1}, \ldots, d_{p, 1}, 0, 0, \ldots, 0, 1).
        \end{align*}
        \item Split Restriction: If a player may perform a split of type $p$ and a split of type $p' > p$ at index $i$, then the player must perform the split of type $p$ rather than the split of type $p'$. This condition makes the Generalized Bergman Game on the $(c, k)$-binacci reccurences (those where $c_i = c_j$ for all $1\leq i,j \leq k$) more accurately reflect a Generalized Zeckendorf Game (see \cite{baird-smithGeneralizedZeckendorfGame2019}) by disallowing moves like the following type two split for $c_1 = c_2 = c_3 = 1$:
        \[
            (0, 0, 0, 2, 1, 0) \to (1, 1, 0, 0, 0, 1).
        \]
        In a proper Generalized Zeckendorf Game where we're far from the boundary, players should instead take the type one splits displayed below:
        \[
            (0, 0, 0, 2, 1, 0) \to (1, 0, 0, 0, 2, 0) \to (1, 1, 0, 0, 0, 1).
        \]
        The reader might then ask why we do not allow only splits of type $1$. This essentially rests on the fact that we wish for the final states of the Generalized Bergman Game to be base-$\beta$ expansions for $\beta$ the dominating root of the given recurrence. For details see \Cref{prop:niplrs-unique-end-state}.
    \end{itemize}
    Once again, players take turns making moves from some starting state $S$, and the game ends when there are no more moves available to play.
\end{defn}

\begin{rmk}
    Let $\beta$ be the unique real root greater than $c_1$ for a given \NIPLRS. Then final states of the Generalized Bergman Game for this \NIPLRS correspond to base-$\beta$ expansions generated by the greedy algorithm, which we show in \Cref{sec:prelim}.
\end{rmk}

\subsection{Summary of Results}\label{subs:results}

We derive a number of results concerning the length of Generalized Bergman Games. To do so, we make heavy use of invariant and monovariant quantities associated to game states and moves. In order to make use of these monovariants, we must also bound the window in which a Generalized Bergman Game may take place. We leverage knowledge of base-$\beta$ representations to extract information about the final game state and, thereby, about the whole game.

To easily exposite our results, we provide the following two definitions.

\begin{defn}\label{defn:niplrs-longest-game}
    Fix some \NIPLRS on which to play the Generalized Bergman Game. Define $\mathcal{M}_\max(S)$ to be the maximum number of moves in a Generalized Bergman Game on a given initial state $S$. Similarly, let $\mathcal{M}_\min(S)$ denote the minimum number of moves in a Generalized Bergman Game on a given initial state $S$.

    We then define
    \begin{align}
        \widehat{\mathcal{M}}_\max(n) \ &\coloneqq\  \max_{\chips{S} = n} \mathcal{M}_\max(S) \label{eq:Mmax}\\
        \widehat{\mathcal{M}}_\min(n) \ &\coloneqq\  \max_{\chips{S} = n} \mathcal{M}_\min(S) \label{eq:Mmin}.
    \end{align}
\end{defn}

\begin{defn}\label{defn:niplrs-window}
    Fix some \NIPLRS on which to play the Generalized Bergman Game. Let $S$ be some initial game state, and without loss of generality suppose that $S$ has leftmost summand at $0$ and rightmost summand at $b$. We define $\mathcal{L}(S)$ to be the absolute value of the minimum index of the leftmost summand throughout any game played with starting state $S$.

    We then define
    \begin{align}
        \widehat{\mathcal{L}}(n, b) \coloneqq \max_{\substack{{}_0S_b \\ \chips{S} = n}} \mathcal{L}(S). \label{eq:Lnb}
    \end{align}
    For convenience, we also let $\widehat{\mathcal{L}}(n) = \widehat{\mathcal{L}}(n, 0)$.
\end{defn}

In order to state the results, we must define asymptotic notation.

\begin{defn}\label{defn:asym}
    Consider two functions $f, g : \N \to \R$. We say that $f(n) = O(g(n))$ provided that there is some constant $K > 0$ and some natural number $M \in \N$ so that for every $n \geq M$ we have $\abs{f(n)} \leq Kg(n)$.

    We may also write this as $g(n) = \Omega(f(n))$. If we have both that $f(n) = O(g(n))$ and $f(n) = \Omega(g(n))$ then we write that $f(n) = \Theta(g(n))$.
\end{defn}

With these notions in mind, we may state the main results of the paper.

\begin{theorem}\label{thm:niplrs-game-length-theta}
    We have $\widehat{\mathcal{M}}_\max(n) = \Theta(n^2)$. That is, the longest possible Generalized Bergman Game for a given \NIPLRS with $n$ summands terminates in $\Theta(n^2)$ moves.
\end{theorem}

\begin{theorem}\label{thm:game-length-theta-lower}
    We have $\widehat{\mathcal{M}}_\min(n) = \Omega(n)$.  That is, the shortest possible game length which is available from all starting states with $n$ summands is $\Omega(n)$.
\end{theorem}

\begin{theorem}\label{thm:bergman-game-length-theta-lower}
    For the Bergman Game in particular, $\widehat{\mathcal{M}}_\min(n) = \Theta(n)$.
\end{theorem}

\begin{theorem}\label{thm:niplrs-game-window}
    We have $\widehat{\mathcal{L}}(n, b) = O(n) + O(b)$ and $\widehat{\mathcal{L}}(n) = \Theta(n)$. That is, the maximum that a summand may be moved left throughout the game is at most linear in the number of summands and the length of the initial window. Furthermore, when the summands in the initial state are concentrated in a single index, we know that this quantity is exactly linear in the number of summands.
\end{theorem}

Together, these theorems nearly characterize termination times and index ranges of Generalized Bergman Games completely, with the exception of an $O(n)$ upper bound for $\widehat{\mathcal{M}}_\min(n)$ in the general case.  (The right bound of any game state is trivially logarithmic in the number of chips and width of initial state, as we show later.)

\section{Preliminaries}\label{sec:prelim}

We establish a number of useful quantities related to the Generalized Bergman Game and outline their properties. Crucially, we connect the game to base-$\beta$ expansions.

\begin{defn}\label{defn:add-game-states}
    Given any two game states $S, T$, define $(S + T)(j) \coloneqq S(j) + T(j)$. If $S(j) \geq T(j)$ for all $j$, we also define $(S - T)(j) \coloneqq S(j) - T(j)$.
\end{defn}

\begin{defn}\label{defn:niplrs-roots}
    Consider some \NIPLRS with coefficients $c_1 \geq \cdots \geq c_k$. Then we define $\beta$ to be the unique real root between $c_1$ and $c_1 + 1$ of the polynomial $x^k - c_1x^{k - 1} - \cdots - c_k$. The existence and uniqueness of such a root is shown in \cite{grabnerGeneralizedZeckendorfExpansions1994}. $\beta$ is then clearly an algebraic integer, and \cite{grabnerGeneralizedZeckendorfExpansions1994} also shows that all of its conjugates $\beta_j$ have modulus less than one.
    
    For convenience, define $\widetilde{\beta}$ to be such a conjugate root of maximal modulus $\widetilde{\beta} < 1$.
\end{defn}

Equipped with this root, we define a key invariant for the Generalized Bergman Game.

\begin{defn}\label{defn:value}
    The value $\val(S)$ of a game state $S$ played with a \NIPLRS with dominating root $\beta$ is $\sum_j S(j)\beta^j$, and the conjugate value $\valt(S)$ of a game state $S$ is $\sum_j S(j)(\widetilde{\beta})^j$.
\end{defn}

\begin{lemma}\label{lemma:value-invariance}
    The value $\val(S)$ and the conjugate value $\valt(S)$ are both invariants throughout the course of the Generalized Bergman Game. That is, if $S$ is some game state and $T$ is a game state obtained by applying some move to $S$, then $\val(S) = \val(T)$ and $\valt(S) = \valt(T)$.
\end{lemma}

\begin{proof}
    We simply need to show that $\sum_j S(j)x^j$ is an invariant for any root of the polynomial $x^k - c_1x^{k - 1} - \cdots - c_k$. 
    
    This sum is clearly invariant under combining moves, because if a combine into index $q$ is performed then we have that
    \begin{align}
        \val(T) - \val(S) = x^{q} - \sum_{j = 1}^k c_jx^{q-j} = x^{q-k}\left(x^k - \sum_{j = 1}^k c_jx^{k-j}\right) = 0.
    \end{align}
    Furthermore, this sum is invariant under all splitting moves. To see this, consider that by simple algebraic manipulation $x$ also satisfies the following equation for all $1 \leq p \leq k - 1$:
    \begin{align}
        (c_p + 1)x^k + \sum_{j = 1}^{p - 1} c_{p - j}x^{k + j} = x^{k + 1} + \sum_{j = 1}^k d_{p, k}x^{k - j}.
    \end{align}
    Therefore if a split of type $p$ is performed at index $q$ we have that
    \begin{align}
        \val(S) - \val(T) = x^{q-k}\left((c_p + 1)x^k + \sum_{j = 1}^{p - 1} c_{p - j}x^{k + j} - x^{k + 1} - \sum_{j = 1}^k d_{p, j}x^{k - j}\right) = 0.
    \end{align}
    This completes the proof.
\end{proof}

With this result, we can begin to think of game states as base $\beta$-representations of their value. Then we reach the following natural realization, of which we take repeated advantage:

\begin{prop}\label{prop:niplrs-unique-end-state}
    Given an initial state $S$ for which the Generalized Bergman Game terminates, the final state $S_f$ of the game is the unique base-$\beta$ expansion of $\val(S)$ obtained via the greedy algorithm.
\end{prop}

\begin{proof}
    The final state $S_f$ provides a base-$\beta$ representation of $\val(S)$ by \Cref{lemma:value-invariance}, and we must have $(S_f(j - k + 1), \ldots, S_f(j)) < (c_k, \ldots, c_1)$ in reverse lexicographic order for every $j$, as otherwise we would either be able to split or to combine. We must show that this is given by the greedy algorithm. Let $b$ be the index of the rightmost summand of $S_f$ and $a$ be the index of the leftmost summand of $S_f$ and let $N_1$ be some integer with $b - (N_1 -3)k < a$. Then
    \begin{align*}
        \val(S_f) - S_f(b)\beta^b &= \sum_{r = 1}^{N_1} \sum_{j = 0}^{k - 1} S_f(b-rk + j)\beta^{b - rk + j} \\
        &< \sum_{r = 1}^{N_1} \beta^{b - (r - 1)k} - \beta^{b - (r - 2)k} < \beta^{b}
    \end{align*}
    where the upper bound on the inner sum derives from the condition of reverse lexicographic ordering.  Specifically, each block of indices of width $k$  has maximized value if it is of the form $(c_k-1,c_{k-1},c_{k-2},\dots,c_1)$.
    This then implies that $v(S_f) < (S_f(b) + 1)\beta^b$, so our state contains the maximal possible number of summands of size $\beta^b$ (i.e., the greedy number). We also cannot have a summand of size $\beta^{b + 1}$ because
    \begin{align*}
        \val(S_f) &= \sum_{r = 1}^{N_1} \sum_{j = 1}^{k} S_f(b - rk +j)\beta^{b - rk + j} \\
                  &< \sum_{r = 1}^{N_1} \beta^{b - (r - 1)k + 1} - \beta^{b - (r - 2)k + 1} < \beta^{b + 1}
    \end{align*}
    Note that we cannot combine or split in the state $S_f - {}_b(S_f(b))$, and so by simple induction $S_f$ is indeed the greedy base-$\beta$ expansion. Because such expansions are unique, the end states of our game are unique for a given value.
\end{proof}

Thus, the final states of our game are base-$\beta$ expansions, and the game actually provides a slow algorithm for computing such expansions exactly given that the game terminates (which we prove in \Cref{sec:niplrs-term}).

We now quickly state and prove how the number of summands changes for each move, as this is one of our basic tools for deriving bounds.

\begin{lemma}\label{lemma:summand-change}
    A combine decreases $\chips{S}$ by $C_0 - 1$, where $C_0 \coloneqq \sum_{j = 1}^k c_j$ and a split does not change $\chips{S}$.
\end{lemma}

\begin{proof}
    This is a simple calculation. Let $S$ be the game state before a given combine into index $q$ and $T$ be the game state after such a combine. Then we have that
    \[
        \chips{S} - \chips{T} = \sum_j S(j) - \sum_j T(j) = -1 + \sum_{j = 1}^k c_j = C_0 - 1.
    \]
    Notice then that splits consist of a single combine played ``in reverse'' and then a combine played ``forward.'' Therefore, a split must preserve the total number of chips.
\end{proof}

\begin{cor}\label{cor:niplrs-combines-bound}
    There are at most $\frac{\chips{S}}{C_0 - 1}$ combines throughout the course of a game on a given initial state $S$. Assuming termination (which we prove in \Cref{sec:niplrs-term}), the number of combines is actually fixed as $\frac{\chips{S} - \chips{S_f}}{C_0 - 1}$, where $S_f$ is the final state of the game.
\end{cor}

\begin{proof}
    If $\chips{S} = 0$, then there are trivially no combines. Now note that if $\chips{S} > 0$ then $\val(S) > 0$, so by \Cref{lemma:value-invariance} we know that $\val(T) > 0$ for any game state $T$ occurring after $S$. This then implies that $\chips{T} > 0$ as well. We then apply \Cref{lemma:summand-change} to provide the given bound on the number of combines. 
    
    The fact that the number of combines is fixed follows from the fact that our final states are base-$\beta$ expansions, which are unique (see \cite{frougnyFiniteBetaexpansions1992}). Thus the number of summands in $S_f$ is fixed.
\end{proof}

\begin{defn}\label{defn:constants}
    For the remainder of the paper, we assume that we are playing the Generalized Bergman Game on a \NIPLRS with coefficients $c_1 \geq \cdots \geq c_k>0$ for $k \geq 2$. We also define a number of constants associated with such a recurrence. These are defined below as well as in \Cref{fig:table-constants}, which the reader may wish to use as a reference table. As convention recall that we say $c_j = 0$ for $j > k$.
    
    We let $d_{p, j} \coloneqq c_j - c_{p + j}$ for $1 \leq j \leq k$ and $1 \leq p \leq k$, $C_r \coloneqq \sum_{j = 1}^k j^rc_j$ for $r \in \Z_{\geq 0}$, $\beta$ be the dominating root of $x^k - c_1x^{k - 1} - \cdots - c_k$, $\widetilde{\beta}$ be the maximal root of modulus less than one, $\rho \coloneqq |\widetilde{\beta}|^{-1}$, and $R \coloneqq \log_\rho \beta$.
\end{defn}

\begin{figure}
    \centering
    \begin{tabular}{c|c}
        Term & Definition \\
        \hline
        $d_{p,j}$ & $c_j-c_{p+j}$ \\ 
        \rule{0pt}{3ex}$C_r$ & $\sum_{j=1}^{k} j^rc_j$ \\
        \rule{0pt}{2ex}$\beta$ & the Dominating Root \\
        $\widetilde{\beta}$ & Maximal Root with $|\widetilde{\beta}| < 1$\\
        $\rho$ & $|\widetilde{\beta}|^{-1}$ \\
        $R$ & $\log_\rho \beta$
    \end{tabular}
    \caption{Table of Useful Constants for a given \NIPLRS with coefficients $c_1 \geq \cdots \geq c_k$ which are used throughout the paper.}
    \label{fig:table-constants}
\end{figure}

\section{The Generalized Bergman Game Terminates}\label{sec:niplrs-term}

In this section we prove that the Generalized Bergman Game is in fact playable.

\begin{prop}\label{prop:niplrs-term}
    Any Generalized Bergman Game terminates in a finite number of moves.
\end{prop}

Using \Cref{cor:niplrs-combines-bound}, it suffices to consider games which consist only of splits in order to prove that the game terminates. To this end, we define a quantity associated to any game state which is a decreasing monovariant in games consisting only of splits. Then, to show that we cannot perform infinitely many splits we bound this monovariant from below by establishing a bound on how far to the left of the initial leftmost summand a particular game can run.

\subsection{Defining and Bounding \texorpdfstring{$\ind(S)$}{the Index Sum}}\label{subs:index-sum}

\begin{defn}\label{defn:index-sum}
    Let $S$ be some game state. We define the index sum of $S$, denoted by $\ind(S)$, to be the quantity $\sum_i iS(i)$.
\end{defn}

\begin{lemma}\label{lemma:niplrs-index-sum}
    Consider a \NIPLRS of depth at least two with coefficients $c_1 \geq \cdots \geq c_k$. A combine into index $i$ increases $\ind(S)$ by $-i(C_0 - 1) + C_1$. Likewise, a split of type $p$ at index $i$ in this game decreases $\ind(S)$ by $p(C_0 - 1)$. Thus, splits decrease the index sum by at least $C_0-1$ and at most $(k-1)(C_0 - 1)$.
\end{lemma}

\begin{proof}
    Consider a combine into index $i$, and let $S_1$ be the game state before this combine and $S_2$ be the game state after this combine. In this case we have that
    \begin{align*}
        \ind(S_2) - \ind(S_1) &= i - \sum_{j = 1}^k c_j(i - j) = -i(C_0 - 1) + C_1.
    \end{align*}
    Now we again view a split of type $p$ at index $i$ as consisting of a reverse combine at index $i$ and a combine into index $i + p$. Then letting $S_1$ be the game state before the reverse combine at index $i$, $S_2$ the game state after the reverse combine at index $i$, and $S_3$ after the combine into index $i + p$ (that is after the split), we have that
    \begin{align*}
        \ind(S_1) - \ind(S_3) &= (\ind(S_1) - \ind(S_2)) + (\ind(S_2) - \ind(S_3)) \\
                              &= -i(C_0 - 1) + C_1 + (i + p)(C_0 - 1) - C_1 \\
                              &= p(C_0 - 1).
    \end{align*}
    The lemma follows by noting that $1 \leq p \leq k - 1$.
\end{proof}

\begin{prop}\label{prop:niplrs-index-sum-bound}
    Suppose that given an initial state ${}_0S_b$ that there is a constant $L_S < 0$ such that for any game state $T$ occuring after $S$ we have that the leftmost summand of $T$ is at an index greater than or equal to $L_S$. Then
    
    $\ind(T) \geq \chips{S}L_S$.
\end{prop}

\begin{proof}
    Note that $\chips{T} \leq \chips{S}$ by \Cref{lemma:summand-change}. We then have that
    \[
        \ind(T) = \sum_i iT(i) \geq \sum_i L_ST(i) = L_S\chips{T} \geq L_S\chips{S}.
    \]
    This establishes the result.
\end{proof}

We now prove a bound $L_S$ and thereby prove termination.

\subsection{Establishing a Left Bound via Gaps}\label{subs:gap-left-bound}

To establish a left bound, we rely on the gaps consisting entirely of zeros between summands in a game state $T$. At a high level, we prove that given an initial state $S$, the gaps in any game state $T$ cannot be longer than some constant length. We then space out the summands of $T$ as much as possible in a worst-case analysis to give a left bound.

\begin{defn}\label{defn:niplrs-max-gap}
    Given a game state $S$, let $\gap(S)$ denote the maximum gap consisting entirely of zeros between summands.
\end{defn}

\begin{rmk}
    Gaps between summands are well studied using probabilistic methods in a variety of integer decompositions, including Zeckendorf Decompositions and \PLRS decompositions (see \cite{summandsGapsGeneralizedZeck,gapsSummandsLatticeZeck,gapsCentralLimitTheorems2016}). For example, the gaps between summands in Zeckendorf decompositions displays geometric decay. It is interesting that they are also a natural quantity to prove termination of the Generalized Bergman Game.
\end{rmk}

We now establish two lemmas concerning the maximum gap size and the location of the rightmost summand of a game state.

\begin{lemma}\label{lemma:niplrs-gap-change}
    A combine may increase the maximum gap size by at most $k$ and a split cannot increase the quantity $\max(\gap(S), k)$, where $k$ is the depth of the recurrence in question.
    
    This then implies that for a game state $T$ occurring in a game after a state $S$ we have
    \[
        \gap(T) \leq \frac{k\chips{S}}{C_0 - 1} + \max(\gap(S), k),
    \]
    where $C_0 := \sum_j c_j$ and the coefficients of the \NIPLRS under question are $c_1 \geq \cdots \geq c_k \geq 1$. 
\end{lemma}

\begin{proof}
    Consider some combining move. In the worst possible case, all of the entries being combined become zeros as below:
    \[
        (1, 0, \ldots, 0, c_1, \ldots, c_k, 1) \to (1, 0, \ldots, 0, 0, \ldots, 0, 2).
    \]
   There are then $k$ zeros added to the gap on the left side, establishing the first piece of the claim.
    
    A split of type $p$ at index $i$ places a nonzero entry $d_{p, j} = c_j - c_{j + p} = c_j \geq 1$ at index $i - j$ for any $k + 1 - p \leq j \leq k$ and a one at index $i + p$. In the worst possible case, all the entries between index $i - k -1 + p$ and index $i + p$ are zero, and the split creates a gap of size $k$ and possibly decreases the size of the gaps to the left of index $i - k$ and to the right of index $i + 1$. This worst case analysis establishes the second part of the claim.
    
    To establish the given bound we then just note from \Cref{cor:niplrs-combines-bound} that the number of combines is bounded by $\frac{k\chips{S}}{C_0 - 1}$.
\end{proof}

\begin{lemma}\label{lemma:niplrs-cannot-shift-game}
    Suppose that $T$ occurs after some state $S$ in a game. Then the rightmost summand of $T$ is located at an index greater than or equal to the leftmost summand of $S$.
\end{lemma}

\begin{proof}
    Let the leftmost summand of $S$ be located at index $a$ and the rightmost summand of $T$ be located at index $b$. We wish to show $a \leq b$

    Consider that $\val(S) = \val(T)$ by \Cref{lemma:value-invariance}, then because $\beta > c_1 \geq 1$ and because $\chips{T} \leq \chips{S}$ we have that\newline
    \[
        \chips{S}\beta^a \leq \val(S) = \val(T) \leq \chips{T}\beta^{b} \leq \chips{S}\beta^b.\\
    \]
    Therefore $a \leq b$ as desired.
\end{proof}

Equipped with these, we can now begin our worst-case analysis which provides the left bound $L_S < 0$ needed to prove that the Generalized Bergman Game terminates.

\begin{prop}\label{prop:niplrs-gap-left-bound}
    Suppose that some game state $T$ occurs after a game state ${}_0S_b$ in a Generalized Bergman Game. Then the leftmost summand of $T$ is located at an index greater than or equal to
    \[
        -\frac{k\chips{S}(\chips{S} - 1)}{C_0 - 1} + \max(\gap(S), k)(\chips{S} - 1) - (\chips{S} - 1)
    \]
    for a given \NIPLRS with coefficients $c_1 \geq \cdots \geq c_k$ and $C_0 := \sum c_j$.
\end{prop}

\begin{proof}
    There are at most $\chips{S}$ summands in $T$, and the rightmost summand of $T$ is located to the right of the zero index by \Cref{lemma:niplrs-cannot-shift-game}. Therefore using the bound on the maximum gap size $\gap(T)$ established in \Cref{lemma:niplrs-gap-change} we may analyze the worst case by placing single summands to the left which are spaced out by $\frac{k\chips{S}}{C_0 - 1} + \max(\gap(S), k)$ zeros. This gives the bound above.
\end{proof}

With this left bound in hand, we can establish \Cref{prop:niplrs-term} (i.e., that any Generalized Bergman Game terminates in a finite number of moves) easily.

\begin{proof}[Proof of \Cref{prop:niplrs-term}]
    Consider some initial state $S$, and note that since the moves of our game are translation invariant we may assume that the leftmost summand of $S$ is located at index zero without loss of generality\footnote{This is an advantage of working with Generalized Bergman Games instead of Generalized Zeckendorf Games, which are not translation invariant}.

    We know that there are finitely many combines performed in the game by \Cref{cor:niplrs-combines-bound}. Therefore, it suffices to show that there cannot be infinitely many splits in a row. 
    
    We know that the index sum decreases by at least one with each split from \Cref{lemma:niplrs-index-sum}. Furthermore, there is a constant $L_S < 0$ depending on the initial state $S$ so that the leftmost summand of any game state $T$ occurring after $S$ is at an index greater than or equal to $L_S$ from \Cref{prop:niplrs-gap-left-bound}. Therefore $\ind(T) \geq L_S\chips{S}$ by \Cref{prop:niplrs-index-sum-bound}.
    
    Putting these results together, there cannot be infinitely many splits or else the index sum would be driven below this bound, which is impossible.
\end{proof}

\begin{rmk}
    Using similar techniques we may establish that a much larger class of games terminates. In particular, the crucial properties used are that all moves either conserve or decrease the number of summands, and that the moves which conserve the number of summands cannot create gaps. We explore this generalization in \Cref{sec:term-mass-cons}.
\end{rmk}

\begin{cor}\label{cor:finite-beta-decomposition}
    Every number $x$ which can be expressed as a finite power series $\sum_{j = a}^b x_j \beta^j$ for $a, b \in \N$ and $x_j \geq 0$ has a finite base-$\beta$ expansion. See \Cref{prop:niplrs-unique-end-state} for the correspondence between end states and base-$\beta$ expansions.
    
    This result is known more generally; \cite{frougnyFiniteBetaexpansions1992} uses ergodic measure theory to prove that for all $\beta$ which are roots of Pisot Recurrences (a superset of \NIPLRS, see \Cref{defn:pisot-recurrence}), all positive members of $\Z[\beta] \subseteq \Z[\beta^{-1}]$ have finite base-$\beta$ representation. 
    
    However, \Cref{prop:niplrs-term} provides an alternate elementary proof that this special class of numbers have finite base-$\beta$ expansions.
\end{cor}

\section{Termination of the Generalized Bergman Game in \texorpdfstring{$O(\chips{S}^2)$}{O(|S| * |S|)} moves}\label{sec:niplrs-n2-term}

We now show that given any initial state $S$, the Generalized Bergman Game terminates in $O(\chips{S}^2)$ moves. To do so, we first establish weaker termination results, which allow us to tackle the main result in much greater concision.


\subsection{The Left Bound for the Generalized Bergman Game}


We first establish a better bound on how far a game may run to the left of the leftmost nonzero index of an initial state $S$. To do this, we must examine the base-$\beta$ expansion of $\val(S)$, or equivalently the final state $S_f$ achieved by playing the game until it ends (which is guaranteed by \Cref{prop:niplrs-unique-end-state} and \Cref{prop:niplrs-term}). 

We begin by slightly generalizing a theorem of Grabner et al. in \cite{grabnerGeneralizedZeckendorfExpansions1994}. For this, we must define a particular type of recurrence relation.

\begin{defn}\label{defn:pisot-recurrence}
    A recurrence $G_n = a_1G_{n - 1} + \cdots + a_kG_{n - k}$ for integers $a_1, \ldots, a_k \geq 0$ and $a_k > 0$ is a Pisot Recurrence provided that the dominating real root $\beta > 1$ of the characteristic polynomial $x^{k + 1} - a_1x^k - \cdots - a_k$ is a Pisot number. That is all the other roots of this polynomial have modulus strictly less than one.
\end{defn}

\begin{theorem}\label{thm:pisot-left-bound}
    Consider some dominating root $\beta$ of the characteristic polynomial of some Pisot Recurrence so that every $\gamma \in \Z[\beta]$ has a finite $\beta$ expansion. 
    
    Then let the base $\beta$ representation of some integer $m$ be $\sum_{\ell = -L(m)}^{R(m)} x_\ell \beta^\ell$ with $x_{-L(m)}, x_{R(m)} \neq 0$. Then $R(m) = \floor{\log_\beta m}$ and $L(m) \leq \log_\rho m + A$ for some constant $A$ and $\beta_2 \neq \beta$ a conjugate of $\beta$ with $\abs{\beta_2} < 1$ which is maximal in modulus among such conjugates with $\rho := \abs{\beta_2}^{-1}$.
\end{theorem}

We apply this theorem to the Generalized Bergman Game by extending it to any element $\alpha \in \Z[\beta]$ with finite base $\beta$ expansion. To do so we must use the following theorem of C. Frougny and B. Solomyak, which we state here for completeness.

\begin{theorem}[Frougny, Solomyak, \cite{frougnyFiniteBetaexpansions1992}]\label{thm:f}
    Let $\beta$ be the dominating root of some \NIPLRS. Then the set $\mathcal{F}$ of numbers with finite base $\beta$ decompositions coincides with the positive members of $\Z[\beta^{-1}]$
\end{theorem}

\begin{porism}\label{porism:niplrs-base-left-bound}
    Consider some $\alpha \in \Z[\beta]$ with a finite base-$\beta$ expansion $\sum_{\ell = -L(\alpha)}^{R(\alpha)} x_\ell \beta^\ell$, where $\beta$ is the dominating root of the characteristic polynomial of some \NIPLRS and $x_{-L(\alpha)}, x_{R(\alpha)} \neq 0$. Then $R(\alpha) = \floor{\log_\beta \alpha}$ and $L(\alpha) \leq \log_\rho \alpha + A$ for some constant $A$ and $\beta_2 \neq \beta$ a conjugate of $\beta$ with $\abs{\beta_2} < 1$ which is maximal in modulus among such conjugates with $\rho := \abs{\beta_2}^{-1}$.
\end{porism}

\begin{proof}
    We replicate the proof in \cite{grabnerGeneralizedZeckendorfExpansions1994} of \Cref{thm:pisot-left-bound} for completeness although the full argument carries through in this greater generality. For convenience let $\beta_j$ denote the conjugates of $\beta$.
    
    The bound for $R(\alpha)$ is trivial, as otherwise $\sum_{-L(\alpha)}^{R(\alpha)} x_\ell \beta^\ell$ would be greater than $\alpha$. To bound $L(\alpha)$, let $M$ be large enough so that $\abs{\alpha\beta_2^M} < 1$, explicitly, we may let $M$ be greater than or equal to $\ceil{\log_\rho \alpha}$. Now use the greedy algorithm in order to write
    \[
        \alpha = \sum_{\ell = -u + 1}^{R(\alpha)} \epsilon_\ell \beta^\ell + \beta^{-u}\gamma^{(u)}
    \]
    for any $u \geq M$ and some positive $\gamma^{(u)} \in \Q(\beta)$ with $\gamma^{(u)} < 1$ for $u > 0$, as otherwise we would have taken another $1 = \beta^0$ using the greedy algorithm.
    
    Now let $\gamma_j^{(u)}$ be the image of $\gamma^{(u)}$ under the isomorphism $\Z[\beta] \to \Z[\beta_j]$ induced by $\beta \mapsto \beta_j$. If $a_1$ is the leading coefficient in the Pisot recurrence, then by the definition of the decomposition we have:
    \[
        \abs{\gamma_j^{(u)}} \leq \abs{\alpha\beta^u_j} + a_1\sum_{\ell = 1}^\infty \abs{\beta_j}^\ell < 1 + \frac{a_1\abs{\beta_j}}{1 - \abs{\beta_j}}
    \]
    Because each $\gamma^{(u)} \in \Z[\beta]$ and all the conjugates $\gamma^{(u)}_j$ are bounded independently of $u$ and $\gamma^{(u)}$, there exist only finitely many such $\gamma^{(u)}$. By \cite{frougnyFiniteBetaexpansions1992} (stated at \Cref{thm:f}) all positive elements of $\Z[\beta^{-1}]$ have finite expansions and $\Z[\beta] \subseteq \Z[\beta^{-1}]$ for the $\beta$ we are considering.
    
    Now choose the maximum length of such an expansion of some $\gamma^{(u)}$ and call it $A'$. When we expand $\gamma^{(u)}$ its nonzero entries will all be in entries less than zero because it is less than one. Therefore the expansion of $\alpha$ will be given by $\sum_{\ell = - u + 1}^{R(\alpha)} \epsilon_\ell \beta^\ell$ and the expansion of $\gamma^{(u)}$ glued on to the left, shifted to have ``zero'' index at index $-u$.
    
    Therefore the bound is $u + A'$. And since any $u \geq M$ works, we have an upper bound of $M + A'$ on $L(\alpha)$. Setting $M := \ceil{\log_\rho \alpha}$ gives $L(\alpha) \leq \log_{\rho} \alpha + A$ as desired, where $A := A' + 1$.
\end{proof}

\begin{cor}\label{cor:niplrs-final-left-bound}
    An initial game state $S$ whose summands are all in non-negative indices represents some element $\val(S) \in \Z[\beta]$ with finite base-$\beta$ expansion by \Cref{prop:niplrs-term} given by the final game state $S_f$. Then the leftmost summand of $S_f$ is at an index greater than or equal to $-\log_\rho \val(S) -A$.
\end{cor}

We now prove \Cref{thm:game-length-theta-lower}.

\begin{proof}
    Consider an initial state $S$ with all summands initially at the same index, without loss of generality $S = {}_0(n)$ and $\val(S) = n$.  By \Cref{cor:niplrs-final-left-bound} and by the trivial logarithmic right bound on the final state, the final state is contained in a window of indices of logarithmic width $a\log(n)$ for some constant $a$ depending on the particular \NIPLRS. 
    
    Then the total number of summands in the final state is at most $ac_1\log(n)$, since any more chips would create opportunity for a split by the pigeonhole principle.  Since we begin with $n$ chips and each combine removes $C_0-1$ chips, then we must perform at least $\frac{n-ac_1\log(n)}{C_0-1}=\Omega(n)$ combines.  The result follows.
\end{proof}

\begin{cor}\label{cor:niplrs-left-bound}
    Consider some initial game state $S$ whose summands are all in non-negative indices. Then the leftmost summand of any game state $T$ occuring after $S$ is at an index greater than or equal to $-k\frac{\chips{S}}{C_0 -1} - \log_\rho \val(S) -A$, where $k$ is the depth of the \NIPLRS under consideration. 
    
    This also shows that $\widehat{\mathcal{L}}(n, b) = O(n) + O(b)$ and $\widehat{\mathcal{L}}(n) = O(n)$ because $\log_\rho \val(S) \leq b + \log_\rho n$ when $S$ has rightmost summand at $b$ and contains $n$ chips. This shows one half of \Cref{thm:niplrs-game-window}.
\end{cor}

\begin{proof}
    Notice that a splitting move cannot push the leftmost summand towards the right, and a combining move can push the leftmost summand at most $k$ to the right, as below:
    \[
        (c_1, \ldots, c_k, 0) \to (0, \ldots, 0, 1).
    \]
    Therefore, since there are at most $\frac{\chips{S}}{C_0 - 1}$ combines throughout the course of the game, we know that the leftmost bound gets pushed at most $k\frac{\chips{S}}{C_0 - 1}$ to the right throughout the course of the game.
    
    In order to satisfy the bound given by \Cref{cor:niplrs-final-left-bound} we must then have the desired bound on the leftmost summand in any state $T$ occuring after $S$.
\end{proof}

Having established the left bound, we now write down the right bound.

\begin{prop}\label{prop:niplrs-right-bound}
    The rightmost summand in a game state $S$ is at an index less than or equal to $\log_\beta \val(S)$. Because $\val(S) = \val(T)$ for any game state $T$ occurring after $S$, this result also holds for game states $T$ occuring after $S$.
\end{prop}

\begin{proof}
    If this were not true, then we would have $\val(S) \geq \beta^{\log_\beta \val(S) + 1}> \val(S)$, which cannot hold.
\end{proof}

This establishes bounds of $-k\frac{\chips{S}}{C_0-1} - \log_\rho \val(S) -A$ and $\log_\beta \val(S)$ where one can play the game, exhibiting that it truly takes place only on a tuple whose length is linear in the number of summands and logarithmic in the value of the game state. In special cases, we can improve on this result.

\begin{prop}\label{prop:phi-final-left-bound}
    Consider a game state $S$ in the Bergman Game whose summands are all in non-negative indices. Then the final game state $S_f$ (i.e., the base-$\varphi$ expansion of $\val(S)$) has leftmost summand at an index greater than or equal to $-\log_\varphi \chips{S} - 2$.
\end{prop}

\begin{proof}
    Let $a$ denote the index of the leftmost summand in $S_f$. Then consider the following chain of inequalities
    \begin{align*}
        \chips{S} &\geq \abs{\valt(S)} = \abs{\valt(S_f)} = \abs{\psi^a + \sum_{j = 1}^\infty S_f(a + j)\psi^{a + j}} \\
                 &\geq \abs{\psi}^a\left(1 - \sum_{j = 2}^\infty S_f(a + j)\abs{\psi}^j\right) \geq \abs{\psi}^a\left(1 - \sum_{j = 0}^\infty \abs{\psi}^{2j + 2}\right) \\
                  &= \varphi^{-a}\left(1 - \frac{\varphi^{-2}}{1 - \varphi^{-2}}\right) = \varphi^{-a - 2}.
    \end{align*}
    By taking the $\log_\varphi$ on both sides and rearranging we then have that $a \geq -2-\log_\varphi \chips{S}$. The only nontrivial inequalities above are in the second line. The first inequality in this line follows by pulling out the $\abs{\psi}^a$, applying the reverse triangle inequality, and noting that $S_f(a + 1) = 0$ because $S_f$ is a reduced game state. The second inequality in this line also follows by considering what we know because $S_f$ is reduced. Namely, each nonzero entry must be followed by a zero, and all entries are less than or equal to one. Then because $j \mapsto \abs{\psi}^j$ is a decreasing function, the maximal arrangement is $(0, 0, 1, 0, 1, 0, 1, 0, \ldots)$; i.e., the sum $\sum_{j = 0}^\infty \abs{\psi}^{2j + 2}$.
\end{proof}

\begin{cor}\label{prop:phi-left-bound}
    Using the arguments of \Cref{cor:niplrs-left-bound}, we can conclude that given an initial state $S$ in the Bergman game, the leftmost summand in any state $T$ occurring after $S$ is at an index greater than or equal to $-2\chips{S} - \log_\varphi \chips{S} - 2$.
\end{cor}

\begin{rmk}
    The proof of this left bound for the Bergman Game uses entirely elementary arguments. The reader might then wonder why we chose to bring in algebraic number theory from both \cite{grabnerGeneralizedZeckendorfExpansions1994} and \cite{frougnyFiniteBetaexpansions1992} for the general case. 
    
    There is a sort of miracle in the Bergman Game case that the conjugate root $\beta_2 = -1/\varphi$ is a negative real number, and so the sign of $\beta_2^j$ is predictable. Furthermore, it relies crucially on the fact that $1 - \frac{\psi^2}{1 - \psi^2} =\psi^2$, a predictable value smaller than one. Because neither of these hold for arbitrary \NIPLRS, the relevant geometric series evaluate to a value possibly larger than one, breaking the argument. 
\end{rmk}


\subsection{The Upper Bound on Termination Time}\label{subs:niplrs-n2-term-upper-bound}


We begin by proving that the game terminates quickly for games within a certain index range. We then use this result to show that we can ``chunk'' apart larger games into non-interacting smaller games.  Finally, we show the game terminates from any initial game state in $n^2$ time where $n := \chips{S}$.

\begin{prop}\label{prop:niplrs-game-length-nonabsolute}
    Let ${}_aS_b$ be some initial game state. Then we know that the Generalized Bergman Game for a \NIPLRS of depth at least two with coefficients $c_1 \geq \cdots \geq c_k \geq 1$ played on this initial game state terminates in
    \begin{align*}
        \frac{k}{2(C_0 - 1)^2}\chips{S}^2 + O(\chips{S}\log \chips{S}) + \chips{S}O(b-a).
    \end{align*}
    For the definitions of these constants refer to \Cref{fig:table-constants}.
\end{prop}

\begin{proof}
    We begin by shifting $S$ so that its leftmost summand is at zero without loss of generality, since our rules are translation invariant. Thus, we can assume that $a = 0$. We now upper bound $\ind(S)$ by $b\chips{S}$, simply by considering the worst case where all of the summands are located at $b$.
    
    By \Cref{cor:niplrs-final-left-bound} we may also lower bound $\ind(S_f)$, where $S_f$ is the final game state. For convenience let $\rho := \abs{\beta_2}^{-1}$ and $R := \log_\rho \beta$ as in the introduction, then we have that
    \begin{align*}
        \ind(S_f) &\geq - \chips{S_f}\log_\rho \val(S) - A\chips{S_f} \geq -\chips{S_f}\log_\rho \chips{S_f}\beta^{b} - A\chips{S_f}\\
        &= -Rb\chips{S_f} - \chips{S_f}\log_\rho \chips{S_f} - A\chips{S_f} \geq -O(\chips{S}\log\chips{S})
    \end{align*}
    where $A$ is the constant dependent on $c_1, \ldots, c_k$ coming from the Algebraic Number Theory in \Cref{porism:niplrs-base-left-bound}.
    
    Quickly recall that $C_0 = \sum c_j$, $C_1 = \sum jc_j$ and note that $C_0 \geq 2$ and $C_1 \geq \frac{k(k-1)}{2}$ when $k \geq 2$. Then consider how much the given combines may increase $\ind(T)$. To do this, we must see that the leftmost possible combine is played into the index $-k\frac{\chips{S}}{C_0 - 1} - \log_\rho \val(S) -A + k$ by \Cref{cor:niplrs-left-bound}, since we need $c_k \geq 1$ summands $k$ indices to the left to combine into an index. The next combine would then have to occur at $-k\frac{\chips{S}}{C_0 - 1} - \log_\rho \val(S) -A + 2k$, and so on until we have all $\frac{\chips{S}}{C_0 - 1}$ combines in order to satisfy the final left bound of \Cref{cor:niplrs-final-left-bound}. Using \Cref{lemma:niplrs-index-sum} we then have that $\ind(T)$ increases by at most
    \[
        \sum_{r = 1}^{\frac{\chips{S}}{C_0 - 1}} \left[-k\frac{\chips{S}}{C_0 - 1} - \log_\rho \val(S) -A + rk - \sum_{j = 1}^k c_j\left(-k\frac{\chips{S}}{C_0 - 1} - \log_\rho \val(S) -A + rk - j\right)\right]
    \]
    over the course of the entire game. We then bound $\log_\rho \val(S)$ above by $Rb + \log_\rho \chips{S}$ and below by $\log_\rho \chips{S}$ to see that $\ind(T)$ increases by at most
    \begin{align*}
        &\sum_{r = 1}^{\frac{\chips{S}}{C_0 - 1}} \left[-k\frac{\chips{S}}{C_0 - 1} - \log_\rho \chips{S} -A + rk - \sum_{j = 1}^k c_j\left(-k\frac{\chips{S}}{C_0 - 1} - Rb - \log_\rho \chips{S} -A + rk - j\right)\right] \\ 
        &=
        \sum_{r = 1}^{\frac{\chips{S}}{C_0 - 1}} \left[k\chips{S} - (C_0 - 1)rk + O(b) + O(\log |S| )\right].
    \end{align*}
    Evaluating this sum gives an increase of
    \begin{align*}
        \Delta^+\ind :&= \frac{k\chips{S}^2}{2(C_0 - 1)} + \chips{S}O(b) + O(\chips{S}\log\chips{S})
    \end{align*}
    in $\ind(T)$ over the course of the whole game. Note then that a split decreases $\ind(T)$ by at least $C_0 - 1$. Therefore in order to not break the lower bound on $\ind(S_f)$ we have the following maximum total number of splits across the course of the entire game:
    \[
        \frac{1}{C_0 - 1} \cdot \left[\Delta^+\ind + Rb\chips{S} + \chips{S}\log_\rho \chips{S} + A\chips{S} + b\chips{S}\right].
    \]
    There are at most $\frac{\chips{S}}{C_0 - 1}$ combines throughout the course of the game, so these get folded into the error term $O(\chips{S}\log\chips{S})$. Unfolding the calculation of $\Delta^+\ind$ then shows the proposition.
\end{proof}

\begin{theorem}\label{thm:niplrs-game-length}
    Given any initial state $S$, the Generalized Bergman Game terminates in $O(\chips{S}^2)$ moves. 
\end{theorem}

\begin{proof}
    Fix some Generalized Bergman Game $G$ which reduces an initial configuration $S$ to its final state $S_f$ on a \NIPLRS with coefficients $c_1 \geq \cdots \geq c_k \geq 1$. For ease of reading, let $\len({}_aT_b) = b - a$ be the length (where we consider ${}_0(n)$ to have length zero) of a game state $T$ with leftmost summand at $a$ and rightmost summand at $b$.
    
    Now let $T^0 \coloneqq S_f$ . We may break $T^0$ into chunks $T^{0, 1}, \ldots T^{0,m_0}$. Namely, whenever we see a run of $k + 1$ or more zeros we break off one chunk. This can be seen below when $k = 2$ for the Bergman Game:
    \begin{center}
        \begin{tabular}{| c c c c c | c c c c | c | c c c | c c c c c c c|}
            & & $T^{0,1}$ & & & & &  & & $T^{0,2}$ & & & & & & & $T^{0,3}$ & & &\\
            \hline
            1 & 0 & 1 & 0 & 1 & 0 & 0 & 0 & 0 & 1 & 0 & 0 & 0 & 1 & 0 & 0 & 1 & 0 & 0 & 1
        \end{tabular}
    \end{center}
    Let $n_{0, j} = \chips{T^{0,j}}$ and note that $T^0 = \sum T^{0,j}$ so $n_0 := \chips{T^0} = \sum n_{0, j}$. Since $T^{0,j}$ contains no runs of $(k + 1)$ zeros we know that $\len(T^{0,j}) \leq (k + 1)n_{0, j}$. 
    
    Now play the game $G$ in reverse from the final state $T^0$ using only reverse splits until we get to the last combine. In doing so, note that all the reverse splits played are independent, and could not have interacted. Why? Well a reverse split always moves the leftmost and rightmost bounds of a game state inward if it doesn't change them. Furthermore, we can't perform a reverse split using things from two different chunks, because reverse splits require nonzero values exactly $k$ apart to perform.
    
    Call the game state after performing these reverse splits and one reverse combine $T^1$. We break $T^1$ into chunks $T^{1,1}, \ldots, T^{1,m_1}$ by keeping track of where each chunk $T^{0, 1}, \ldots, T^{0, m_0}$ went under these game moves, combining two chunks if they become closer than a gap of $(k + 1)$ zeros.
    
    We then note that $m_1 = m_0$ or $m_1 = m_0 - 1$. Reverse splits drive the bounds of each chunk inwards and cannot merge chunks. Then a reverse combine may only bridge a single gap of at most $2k$ zeros between chunks. Similarly, a reverse combine cannot create new chunks, and can only increase the length of a chunk by at most $k$.
    
    Continue this process to form a sequence of game states $T^0, \ldots, T^{K}$ where $K$ is the number of combines, and chunks $T^{i, j}$ where $0 \leq i \leq K$ and $1 \leq j \leq m_i$. We then backtrack through any remaining reverse splits to get chunks $S^1, \ldots, S^{M}$ of the initial state. These chunks all play independently in the game $G$ by the previous remarks.
    
    We now begin to bound the length of the game $G$. To do this, we bound the total lengths of the chunks in a clever way. A reverse split never increases the length of a chunk, and a reverse combine increases the total length of chunks by at most $2k$. Therefore since there are at most $\frac{\chips{S}}{C_0 - 1}$ combines we have that
    \[
        \sum_i \len(S^i) \ \leq\  \frac{2k\chips{S}}{C_0 - 1} + \sum_j \len(T^{0, j}) \ \leq\  \frac{2k\chips{S}}{C_0 - 1} + (k + 1)\sum_j n_{0, j} \ = \ O(\chips{S}).
    \]
    Because the game plays independently on each chunk $S^1, \ldots, S^M$, we may use the bound from \Cref{prop:niplrs-game-length-nonabsolute} to see that the number of moves is bounded above by the following:
    \begin{align*}
        \sum_i \left[O(\chips{S^i}^2) + \chips{S^i}O(\len(S^i))\right].
    \end{align*}
    Using the fact that $\sum_i \chips{S^i} = \chips{S}$ and we may then bound each piece of this easily:
    \begin{align*}
        \sum_i \chips{S^i}^2 &\leq \left(\sum_i \chips{S^i}\right)^2 = \chips{S}^2 \\
        \sum_i \len(S^i)\chips{S^i} &\leq \chips{S}\sum_i \len(S^i) = O(\chips{S}^2).
    \end{align*}
    In total, this provides a bound on the length of a game $G$ with initial state $S$ by $O(\chips{S}^2)$ moves, completing the proof.
\end{proof}


\section{Extremal Games}\label{sec:niplrs-long-game}

We present an $\Omega(n^2)$ move GBG game and thus show that our upper bound on $\widehat{\mathcal{M}}_\max(n)$ is tight.  We subsequently show that our bound is tight in the constant of the leading term for the Bergman Game.  We then present an $O(n)$ Bergman Game, which shows that our lower bound on $\widehat{\mathcal{M}}_\min(n)$ is tight for the Bergman Game.


\subsection{Construction of a \texorpdfstring{$\Omega(\chips{S}^2)$}{Omega(|S|*|S|)} Generalized Bergman Game}
We show that there is a long Generalized Bergman Game. In particular, we show that using a particular strategy for selecting moves, we can make a game which takes $\Omega(n^2)$ moves from the initial state $S := {}_0(n)$ to the final state $S_f$. This proves the asymptotic tightness of \Cref{thm:niplrs-game-length}, showing that games are in fact $\Theta\left(\chips{S}^2\right)$ as claimed in \Cref{thm:niplrs-game-length-theta}. We start by detailing the strategy employed in this game.

\begin{defn}\label{defn:slcr}
    The Split Left and Combine Right ($\slcr$) strategy mandates that each player always takes the split farthest to the left whenever there is a split available, and otherwise takes the farthest combine to the right.
    
    Similarly, we may define the strategies $\slcl$, $\srcl$, $\srcr$ as well as the strategies $\clsl$, $\clsr$, $\crsr$, and $\crsl$.
\end{defn}

\begin{prop}\label{prop:slcr-phases}
    The $\slcr$ strategy played on any initial state divides the game into two phases. In Phase I, all of the splits are performed, and in Phase II only combines are performed.
\end{prop}

\begin{proof}
    Suppose that we are in a state $S$ where we cannot split. We show that using the $\mathcal{S}_{\mathscr{L}}\mathcal{C}_{\mathscr{R}}$ strategy we never split again, proving the first claim. 
    
    In this case, we know that no string of entries in $S$ is pairwise greater than or equal to $(c_p + 1, c_{p - 1}, \ldots, c_1)$. In order to split again, we would need to combine to create an entry of size at least $c_p + 1$ occuring before entries of size at least $c_{p - 1}, \ldots, c_1$. The only possible way to do this would be to take the move below
    \[
        (c_k, c_{k - 1}, \ldots, c_1, c_p, c_{p - 1}, \ldots, c_1) \to (0, 0, \ldots, c_p + 1, c_{p - 1}, \ldots, c_1).
    \]
    However, the $\slcr$ strategy mandates that we instead take the move
    \[
        (c_k, c_{k - 1}, \ldots, c_1, c_p, c_{p - 1}, \ldots, c_1) \rightarrow (d_{p, k}, \ldots, d_{p, 1}, 0, 0, \ldots, 0, 1).
    \]
    When we take this move, all strings remain pairwise less than or equal to $(c_p + 1, c_{p - 1}, \ldots, c_1)$, and so we can never split again, dividing the game played with the $\slcr$ strategy into Phase I and Phase II. 
\end{proof}

\begin{cor}\label{cor:slcr-slcl}
    Due to this division into phases, the $\slcl$ strategy agrees with the $\slcr$ strategy for all of the moves in Phase I. As a corollary, the $\slcl$ strategy takes at least as many moves as the $\slcr$ strategy.
\end{cor}

\begin{proof}
    Because the $\slcl$ strategy agrees with the $\slcr$ strategy when only splits are played, the game proceeds through Phase I with no differences.
    
    Because the number of combines played in a game on an initial state $S$ is constant (see \Cref{cor:niplrs-combines-bound}), we have the desired claim that the $\slcl$ strategy takes at least as many moves as the $\slcr$ strategy for any initial state. In particular, when an extra split is performed in the second phase for the first strategy, it is strictly longer.
\end{proof}

\begin{theorem}\label{thm:niplrs-long-game-initial}
    The $\slcr$ strategy takes $\Omega(n^2)$ moves from the initial state ${}_0(n)$. 
    
    In particular, letting $c_1, \ldots c_k$ be the coefficients of the \NIPLRS under question and $\beta$ be the dominating real root, the Generalized Bergman Game takes at least $\frac{n(n - 2c_1\log_\beta n + 1)}{2(k-1)c_1(C_0 - 1)}$ splits for any $n$. Furthermore, the leftmost summand in the state between Phase I and Phase II is located at an index less than or equal to $-\frac{n}{c_1} + \log_\beta n + 1$, showing that $\widehat{\mathcal{L}}(n) = \Omega(n)$.
\end{theorem}

\begin{proof}
    Let $T$ be the state of the game after the completion of Phase I and before the first combine by \Cref{prop:slcr-phases}. Note first that all the entries of $T$ are less than or equal to $c_1$ because we have finished Phase I. Furthermore, the rightmost summand of $T$ is located at an index less than or equal to $\log_\beta \val(T) = \log_\beta \val({}_0(n)) = \log_\beta n$ by \Cref{prop:niplrs-right-bound}. Furthermore $\chips{T} = n$ because splits preserve the number of chips, and we have performed only splits in Phase I.
    
    We now recall the index sum $\ind(T)$ introduced in \Cref{sec:niplrs-term}. Note that a split always decreases the index sum by at least $C_0 - 1$ and at most $(k - 1)(C_0 - 1)$ (see \Cref{lemma:niplrs-index-sum}). We then know the following upper bound on $\ind(T)$ given by placing $c_1$ summands in each index starting at $\log_\beta n$ and moving left, and placing $n$ summands total:
    \[
        \ind(T) \leq \sum_{j = -\frac{n}{c_1} + \log_\beta n + 1}^{\log_\beta n} c_1j = \frac{n(2c_1\log_\beta n - n + 1)}{2c_1}.
    \]
    Therefore since $\ind({}_0(n)) = 0$, we know that there must have been at least  $\frac{-\ind(T)}{(k-1)(C_0-1)}$ splits. Calculating this based on our upper bound on $\ind(T)$ we see that
    \[
        \# \text{ of splits} \geq \frac{n(n - 2c_1\log_\beta n + 1)}{2(k-1)c_1(C_0 - 1)}.
    \]
    This completes the proof.
\end{proof}


\subsection{Improving the Coefficient on the Dominating Term for the Bergman Game}\label{subs:bound-iteration}

In this subsection, we prove that for the Bergman Game, the coefficients on the dominating term coming from \Cref{thm:niplrs-game-length} and \Cref{thm:niplrs-long-game-initial} agree. That is, we can strengthen \Cref{thm:niplrs-game-length-theta} to the following in this special case.

\begin{prop}\label{prop:phi-game-length-theta}
    For the Bergman Game, \Cref{prop:niplrs-game-length-nonabsolute}--which bounds the length of the longest Bergman game played on ${}_0(n)$ above by $n^2 + 2n\log_\varphi n + n$--is essentially tight.
    
    More precisely, for $n \geq 20$ the length of the longest Bergman game played on ${}_0(n)$ lies in the interval
    \[
        \left[n^2 - 6n -3n\log_\varphi n + 5\log_\varphi n + 2(\log_\varphi n)^2 - 1 , n^2 + 2n\log_\varphi n + n\right].
    \]
    Note specifically the agreement in the coefficient of the dominating terms for each side of the interval.
\end{prop}

We can also strengthen \Cref{thm:niplrs-game-window} in the case of the Bergman Game.

\begin{prop}\label{prop:phi-left-bound-theta}
    \Cref{prop:phi-left-bound}, which tells us that the maximum distance $\mathcal{L}(n)$ that the leftmost summand can move to the left in a Bergman Game on $n$ concentrated summands is bounded above by $2n + \log_\varphi n + 2$ is essentially tight.
    
    More precisely, for $n \geq 20$ we have that this quantity $\mathcal{L}(n)$ lies between $2n - 3\log_\varphi n - 8$ and $2n + \log_\varphi n + 2$.
\end{prop}

As expected from the methods in \Cref{sec:niplrs-long-game}, these propositions are intrinsically linked. In order to prove these statements, we predict the left edge of the game state $T$ occurring after Phase I of the $\slcr$ strategy. The broad outline of the argument consists of two iterative processes which feed into each other. The first iterative process takes in the position of the leftmost summand and guarantees that the left edge has a certain form with a certain number of zeros. Knowing this number of zeros, we can push the position of the leftmost summand further to the left, and then we run the first process as many times as necessary.

To ease reading, we divide this into a few lemmas

\begin{lemma}\label{lemma:phi-easy-edge}
    Let $T$ be some game state consisting entirely of ones and zeros occurring after ${}_0(n)$ in the Bergman Game. Further require that the leftmost summand of $T$ is at a position to the left of index $-\log_\varphi n - 5$. Then the left edge of $T$ must have the form $110\ldots$.
\end{lemma}

\begin{proof}
    We know that eventually this leftmost summand must be pushed right via a combine, as otherwise we would not satisfy the left bound of $-\log_\varphi n-2$ on the final state coming from \Cref{prop:phi-final-left-bound}. Using a $\slcr$ strategy on $T$ gives a game consisting of only combines, and so because combines always move summands right, $T$ must already have a combine at the leftmost edge available. That is we know the left edge of $T$ takes the form $11\ldots$.
    
    We now rule out the possibility that three or more ones show up at the left edge in order to finish the proof. Let $T = \underbrace{11\ldots1}_{s \text{ ones}}0T^r$, where $T^r$ is some sub-state of $T$ and $s \geq 3$. First we will perform all possible combines available in the $s$ ones. To see what state then occurs, we case out on whether $s$ is even or odd.
    \begin{itemize}
        \item Suppose that $s$ is odd. For the sake of simplicity, we work out the case where $s = 5$:
            \begin{center}
                \begin{tabular}{cccccc}
                    1 & 1 & 1 & 1 & 1 & 0\\
                    1 & 1 & 1 & 0 & 0 & 1 \\
                    1 & 0 & 0 & 1 & 0 & 1 
                \end{tabular}
            \end{center}
            Notice that after the first move we are simply taking the available combines for $s = 3$. A simple inductive argument shows that using the $\slcr$ strategy on these types of states has the following effect:
            \[
                \underbrace{1111\ldots1}_{s \text{ ones}}0 \;\; \xmapsto{\slcr} \;\; 100\underbrace{101\ldots01}_{(s - 1)/2 \text{ ones}}.
            \]
            The leftmost summand is not moved in by doing these moves. Further, doing these moves results in a game state $100101\ldots01T^r$ with only zeros or ones. Applying the first piece of this lemma we should have left edge $11\ldots$ in this new state. Clearly this is not the case so $s$ cannot be odd.
        \item Suppose that $s$ is even. For the sake of simplicity, we work out the case where $s = 4$:
            \begin{center}
                \begin{tabular}{ccccc}
                    1 & 1 & 1 & 1 & 0\\
                    1 & 1 & 0 & 0 & 1 \\
                    0 & 0 & 1 & 0 & 1 
                \end{tabular}
            \end{center}
            Notice that after the first move we are simply taking the available combines for $s = 2$. A simple inductive argument shows that using the $\slcr$ strategy on these types of states has the following effect:
            \[
                \underbrace{111\ldots1}_{s \text{ ones}}0 \; \; \xmapsto{\slcr} \; \; 00\underbrace{101\ldots01}_{s/2 \text{ ones}}.
            \]
            The leftmost summand is moved in by two when we do these moves. Further, doing these moves results in a game state $0010101\ldots01T^r$ with only zeros or ones. Applying the first piece of this lemma because $-\log_\varphi n - 3 < -\log_\varphi n - 2$ we should have left edge $11\ldots$ in this new state. Clearly this is not the case when $s > 2$.
    \end{itemize}
    This shows that the left edge must be exactly $110\ldots$ as claimed.
\end{proof}

\begin{lemma}\label{lemma:phi-gen-zeros}
    Let $T$ be some game state consisting entirely of ones and zeros occurring after ${}_0(n)$ in the Bergman Game. Further require that the leftmost summand of $T$ is at a position to the left of index $-L \leq -\log_\varphi n - 3$. Then the left edge of $T$ must have the form 
    \[
        \underbrace{110101\ldots10}_{z_L \text{ zeros}}\ldots
    \]
    where $z_L = \floor{\frac{L - \log_\varphi n - 3}{2}}$. Furthermore all these zeros occur in negative indices because $-L + 2z_L \leq -\log_\varphi n - 3 < 0$.
\end{lemma}

\begin{proof}
    We essentially apply \Cref{lemma:phi-easy-edge} $z_L$ times. Namely if $1 \leq z_L$, then this implies that $2 \leq L - \log_\varphi n - 3$, so there is a summand at a position left of index $-L \leq -\log_\varphi n - 5$. We apply \Cref{lemma:phi-easy-edge} and write $T = 110T^r$, then taking the available combine move to get $T' = 001T^r$. Note that $T'$ consists entirely of ones and zeros and has leftmost summand left of $-L + 2$. If $-L + 2 \leq -\log_\varphi n - 5$ we again apply \Cref{lemma:phi-easy-edge} to predict that $T' = 00110\ldots = 001T^r$, giving that $T^r$ has left edge $10\ldots$. We're then able to see that $T$ has left edge $11010\ldots$.
    
    Running this argument for all $z \in \N$ with $-L + 2z \leq -\log_\varphi n - 5$ then gives the desired result.
\end{proof}

\begin{lemma}\label{lemma:phi-left-bound-iter}
    Let $T$ be some game state consisting entirely of ones and zeros occuring after ${}_0(n)$ in the Bergman Game such that $\chips{T} = n$. Further require that the leftmost summand of $T$ is at a position to the left of index $-L \leq -\log_\varphi n - 3$. Then there is a summand to the left of index $-n - z_L + \log_\varphi n + 1$.
\end{lemma}

\begin{proof}
    Notice that there are at most $\log_\varphi n + 1$ summands in non-negative indices because of the right bound of $\log_\varphi n$ proved in \Cref{prop:niplrs-right-bound}. Applying \Cref{lemma:phi-gen-zeros} we know that the left edge of $T$ has the form
    \[
        \underbrace{110101\ldots10}_{z_L \text{ zeros}}\ldots.
    \]
    All of these zeros occur in negative indices. There are then $n - \log_\varphi n - 1$ summands to be distributed in the negative indices and there are at least $z_L$ zeros in these indices. By the pigeonhole principle there is a summand to the left of index $-n - z_L + \log_\varphi n + 1$ as desired.
\end{proof}

\begin{prop}\label{prop:phi-left-edge}
    Pick $n \geq 20$. Then we have that the left edge of the game state $T$ achieved by playing Phase I of the $\slcr$ strategy on ${}_0(n)$ has the form
    \[
        \underbrace{1101\ldots10}_{z \text{ zeros}}\ldots,
    \]
    where each zero is in a position left of index $-\log_\varphi n - 3$ and $z := n - 2\log_\varphi n - 7$.
\end{prop}

\begin{proof}
    Let $L^{(1)} \coloneqq n - \log_\varphi n - 1$. Then in the spirit of \Cref{lemma:phi-left-bound-iter} recursively define
    \begin{align*}
        z^{(j)} &\coloneqq \frac{L^{(j)} - \log_\varphi n - 5}{2} \tag{$j \geq 1$} \\
        L^{(j)} &\coloneqq n + z^{(j-1)} - \log_\varphi n - 1 \tag{$j \geq 2$}.
    \end{align*}
    The key here is that $z^{(j)} \leq z_{L^{(j)}}$ by the definition of the latter term in \Cref{lemma:phi-gen-zeros}. We then define $L^{(j)}$ in terms of $z^{(j - 1)}$ using the result of \Cref{lemma:phi-left-bound-iter}, motivated by the idea that we will always have a summand to the left of $L^{(j)}$.
    
    We now explicitly compute that
    \begin{align*}
        L^{(j)} &= n + \frac{L^{(j - 1)}}{2} - \frac{3\log_\varphi n}{2} - \frac{7}{2} \\
        L^{(j)} - L^{(j - 1)} &= \frac{L^{(j - 1)} - L^{(j - 2)}}{2} \\
        L^{(J)} - L^{(1)} &= \sum_{j = 0}^{J - 2} L^{(j + 2)} - L^{(j + 1)} = \left(2 - \frac{1}{2^{J - 2}}\right)\left(L^{(2)} - L^{(1)}\right) \\
        L^{(J)} &= \left(2 - \frac{1}{2^{J - 2}}\right)\left(\frac{n}{2} - \log_\varphi n - 3\right) + n - \log_\varphi n - 1 \\
        &= \left(2 - \frac{1}{2^{J - 1}}\right)n - \left(3 - \frac{1}{2^{J - 2}}\right)\log_\varphi n - \left(7 - \frac{3}{2^{J - 2}}\right).
    \end{align*}
    Notice that for $n \geq 20$ we have for all $j$ that $L^{(j)} \geq \log_\varphi n + 3$. Thus, we can apply \Cref{lemma:phi-gen-zeros,lemma:phi-left-bound-iter} repeatedly as many times as we wish to obtain the desired left edge and a summand to the left of $L^{(j)}$ for every $j$. As $j \to \infty$ we have that $L^{(j)}$ approaches $2n - 3\log_\varphi n - 7$ monotonically for $n \geq 20$. Likewise, $z^{(j)}$ approaches $n - 2\log_\varphi n - 6$ monotonically for $n \geq 20$ as $j \to \infty$.
    
    Therefore since the true number of zeros must be an integer, we know that there must be at least $n - 2\log_\varphi n - 7$ zeros as claimed, and they must be arranged as above by the specific claims about the left edge made in \Cref{lemma:phi-gen-zeros}.
\end{proof}

We can then immediately conclude from the pigeonhole principle that there is a summand in $T$ to the left of index $-2n + 3\log_\varphi n + 8$, confirming \Cref{prop:phi-left-bound-theta}.

Furthermore, we can also immediately conclude \Cref{prop:phi-game-length-theta} using the techniques of \Cref{sec:niplrs-long-game}. Namely, we know that in $T$ there are summands $110101\ldots01$ at indices to the left of $-2n + 3\log_\varphi n + 8$ and $-2n + 3\log_\varphi n + 7 + 2j$ for $1 \leq j \leq n - 2\log_\varphi n - 7$. This implies that $-\ind(T)$ is bounded below by
\[
    2n - 3\log_\varphi n -8 + \sum_{j = 1}^{n - 2\log_\varphi n - 7} \left(2n - 3\log_\varphi n - 7 -2j\right) - \log_\varphi n.
\]
Evaluating this sum gives
\[
    -\ind(T) \geq n^2 - 6n -3n\log_\varphi n + 5\log_\varphi n + 2(\log_\varphi n)^2 - 1.
\]
We know that each split decreases the index sum by exactly one in the Bergman Game, and so this tells us that there are at least this many splits, implying \Cref{prop:phi-game-length-theta}.

\begin{rmk}
    We are interested in generalizing this argument to other \NIPLRS. The difficult mainly comes from generalizing \Cref{lemma:phi-easy-edge}. Here we were able to make a great degree of progress simply by ruling out the possibility that we have all ones because in the Bergman game there is a binary choice for the value of each entry. For other \NIPLRS, the arguments become much more involved as there are $c_1$ choices for the value of each entry in $T$.
\end{rmk}

\subsection{Construction of an $O(\chips{S})$ Bergman Game}\label{subs:bergman-short-game}

Here, we construct a $O(\chips{S})$ Bergman Game given arbitrary initial state.  We thus establish a tight bound $\Theta(n)$ on the shortest number of moves which can ensure completion of a game with $n = \chips{S}$ summands arranged in any initial state, for all $n$, and thus we prove \Cref{thm:bergman-game-length-theta-lower}.  We play this game via the following strategy.

\begin{defn} \label{def:quick-term-strat}
    The Quick Termination ($\mathcal{Q}\mathcal{T}$) strategy instructs the players to choose moves in the following order of priority.
    \begin{enumerate}
        \item Combine.
        \item Split anywhere with more than two summands.
        \item Split at the 2 in a configuration $(\dots,a_1,a_2,0,2,0,a_3,\dots), a_1>0 \text{ or } a_3>0$.
        \item Split at the 2 in a configuration $(\dots,0,a_2,0,2,0,0,\dots), a_2\geq 2$.
        \item Split at the rightmost possible index.
    \end{enumerate}
\end{defn}

\begin{prop}
    The $\mathcal{Q}\mathcal{T}$ strategy always produces a game  of length at most $4\chips{S}$ moves from any starting state with $\chips{S}$ summands.
\end{prop}

\begin{proof}
    Let $n=\chips{S}$.  We may only combine at most $n$ times.  Furthermore, when we take (2) or (3) in \Cref{def:quick-term-strat}, we immediately create opportunity for a combine.  Additionally, taking (4) creates an index with three summands, at which we then split (3) and subsequently combine.  Therefore, doing either (1), (2), (3), or (4) causes us to combine within the next three moves.  As such, we may only make $3n$ total of these moves.
    
    We now bound the number of (5) moves.  By construction, the game state $S$ which precedes a (5) move contains only 0's, 1's, and 2's and no adjacent summands, and every 2 is surrounded by a neighborhood $(\dots,0,a_2,0,2,0,0,\dots)$, where $a_2<2$.  Then the neighborhood of any 2 will look like $(a_1,0,0,1,0,1,0,\dots,1,0,1,0,2,0,0)$.  We call these configurations $A_k$, where $k$ is the number of non-zero indices in the configuration, excluding $a_1$.  Call the configuration $(1,0,1,0,\dots,1,0,1)=B_k$, where again $k$ refers to the number of non-zero indices.  In general then, any game state just prior to making a (5) move is a finite string of disjoint $A_k$'s and $B_k$'s, padded with possible intervening zeros.
    
    By \Cref{def:quick-term-strat}, we first split at the right-most possible index when we make a (5) move, which is also the 2 of the right-most $A_k$.  We do not create an opportunity for moves (1)-(4) but rather create an $A_{k-1}$ two indices to the left of the original $A_k$, so inductively, we make $k$ splits in a row from right to left and convert our $A_k$ into a state $(a_1,1,0,0,1,0,1,0,\dots,1,0,1,0,1,0)$.  Regardless of the value of $a_1$, we may perform no more (5) moves in this region of $S$; depending on $a_1$, we may play one more combine, and then we will have no more moves in this region.  Therefore, playing on each $A_k$ provides at most $k$ consecutive (5) moves.  Furthermore, these moves do not alter any part of $S$ to the right of $A_k$.  Therefore, once we play our $k$ splits on an $A_k$, we never play on any indices to the right of its left edge for the remainder of the game.
    
    Then the total number of (5) moves is exactly the total number of split moves played on all the $A_k$ which ever form. Each $A_k$ gives $k$ (5) moves, which is less than $k+1$ the number of summands in an $A_k$.  Then the total number of option (5) moves is less than the total number of summands in all the $A_k$, which is at most $n$.  We have $3n$ moves from options (1)-(4) and $n$ from (5), giving a linear bound of $4n$ total moves from any initial state.
\end{proof}

We therefore deduce that for the Bergman Game, $\hat{B'}(n)=O(n)$.  Together with \Cref{thm:game-length-theta-lower}, \Cref{thm:bergman-game-length-theta-lower} follows: in the Bergman case $\hat{B'}(n) = \Theta(n)$.

\section{Future Work}\label{sec:future}

Our research provides a number of directions for possible future work, roughly divided into five categories: generalizing bound iteration, proving claims about specific strategies, improving the left bounds for general \NIPLRS, game theoretic concerns, the behavior of random games, and defining a game on general \PLRS.

\subsection{Improving Bounds and Exploring Specific Strategies}

A generalization of bound iteration would entail closing the gap between the coefficients on the dominating terms for both the left bound and the termination time (see \Cref{cor:niplrs-left-bound} and \Cref{prop:niplrs-game-length-nonabsolute} respectively). The bound iteration process detailed in \Cref{subs:bound-iteration} allows us to do this in the case of the Bergman Game. However, the argument heavily relies on the fact that after Phase I of the $\slcr$ strategy we know that there are only zeros or ones in play. For other \NIPLRS, there may be entries ranging from zero to $c_1$, which creates more variability to account for in the arguments.

While we are able to demonstrate strategies which terminate in $\Theta(n^2)$ time for general PLRS, we only provide a $\Theta(n)$ game for the Bergman Game (see \Cref{sec:niplrs-long-game} and \Cref{subs:bergman-short-game}).  Future work may generalize the latter.  We have also not been able to show that any particular strategy takes the most moves or the least moves. Showing such a statement would likely rely on the quantities like $\chips{S}$ and $\ind(S)$ that we have established here, but would almost certainly require a host of new ideas as well.

To improve the left bounds for general \NIPLRS would entail removing the dependence on the value of a game state $S$ inherent in the bound on the index of the leftmost summand for the final state proved via Algebraic Number Theory (see \Cref{cor:niplrs-final-left-bound}). Instead one may be able to prove a bound similar to the one demonstrated in the special case of the Bergman Game (see \Cref{prop:phi-final-left-bound}) which depends only on the number of summands $\chips{S}$. This would allow us to improve the coefficients on the dominating $\chips{S}^2$ term for the bound on the game length in \Cref{thm:niplrs-game-length}. Alternatively, a counterexample to this claim could be found, which would further solidify the special properties of $\varphi$ that crop up when proving things about the Bergman Game.

\subsection{Game Theoretic Concerns, Random Games, and \PLRS games}

As far as game theoretic concerns. We are interested in showing that a winning strategy exists for either player--either for the Bergman Game or in general. Furthermore, ideally the existence of such strategies would depend on the initial state in a predictable or computable way. We expect that a proof that either player has a winning strategy would be non-constructive, as in the case of the Zeckendorf Game itself (see \cite{baird-smithZeckendorfGame2020}). The strategy-stealing techniques used to show the existence of a winning strategy for the Zeckendorf Game--at least by direct examination of the game tree--seem untenable due to the quickly growing number of unique states observed throughout a game. See \Cref{fig:bergman-game-tree} for an example of such a tree going to a depth of three moves in the case of the Bergman Game. Every increase in depth from this point forward will spawn so many distinct nodes that the argument from \cite{baird-smithZeckendorfGame2020} becomes intractable.

\begin{figure}[ht]
    \centering
    \scalebox{0.75}{\begin{tikzpicture}[cell/.style={rectangle,draw=black},
space/.style={minimum height=0.9em,matrix of nodes,row sep=-\pgflinewidth,column sep=-\pgflinewidth,column 1/.style={font=\ttfamily}}, scale=0.5, every node/.style={transform shape}]
        \matrix (first) [space, column 1/.style={nodes={cell,minimum width=2em}}]
        {
            $n$ \\
        };
        \matrix (second) [below=of first, space, column 1/.style={nodes={cell,minimum width=2.em}}, column 2/.style={nodes={cell,minimum width=2em}}, column 3/.style={nodes={cell,minimum width=2em}}, column 4/.style={nodes={cell,minimum width=2em}}]
        {
            1 & 0 & $n - 2$ & 1 \\
        };
        \draw[->, red] (first) -- (second);
        
        \matrix (thirdc) [below right=of second, space, column 1/.style={nodes={cell,minimum width=2em}}, column 2/.style={nodes={cell,minimum width=2em}}, column 3/.style={nodes={cell,minimum width=2em}}, column 4/.style={nodes={cell,minimum width=2em}}, column 5/.style={nodes={cell,minimum width=2em}}]
        {
            1 & 0 & $n - 3$ & 0 & 1 \\
        };
        
        \matrix (thirds) [below left=of second, space, column 1/.style={nodes={cell,minimum width=2em}}, column 2/.style={nodes={cell,minimum width=2em}}, column 3/.style={nodes={cell,minimum width=2em}}, column 4/.style={nodes={cell,minimum width=2em}}]
        {
            2 & 0 & $n - 4$ & 2  \\
        };
        
        \draw[->, blue] (second) -- (thirdc);
        \draw[->, red] (second) -- (thirds);

        \matrix (thirdsc) [below=of thirdc, space, column 1/.style={nodes={cell,minimum width=2em}}, column 2/.style={nodes={cell,minimum width=2em}}, column 3/.style={nodes={cell,minimum width=2em}}, column 4/.style={nodes={cell,minimum width=2em}}, column 5/.style={nodes={cell,minimum width=2em}}]
        {
            2 & 0 & $n - 5$ & 1 & 1 \\
        };
        
        \matrix (thirdss) [below=of thirds, space, column 1/.style={nodes={cell,minimum width=2em}}, column 2/.style={nodes={cell,minimum width=2em}}, column 3/.style={nodes={cell,minimum width=2em}}, column 4/.style={nodes={cell,minimum width=2em}}, column 5/.style={nodes={cell,minimum width=2em}}]
        {
            3 & 0 & $n - 6$ & 3 \\
        };
        
       \matrix (thirdssl) [below left=of thirds, space, column 1/.style={nodes={cell,minimum width=2em}}, column 2/.style={nodes={cell,minimum width=2em}}, column 3/.style={nodes={cell,minimum width=2em}}, column 4/.style={nodes={cell,minimum width=2em}}, column 5/.style={nodes={cell,minimum width=2em}}, column 6/.style={nodes={cell,minimum width=2em}}]
       {
       1 & 0 & 0 & 0 & $n - 4$ & 2 \\
       };
        
       \matrix (thirdssr) [below right=1cm and 0.5cm of thirds, space, column 1/.style={nodes={cell,minimum width=2em}}, column 2/.style={nodes={cell,minimum width=2em}}, column 3/.style={nodes={cell,minimum width=2em}}, column 4/.style={nodes={cell,minimum width=2em}}, column 5/.style={nodes={cell,minimum width=2em}}]
       {
           2 & 1 & $n - 4$ & 0 & 1\\
       };
       
        \draw[->, red] (thirdc) -- (thirdsc);
        \draw[->, red] (thirds) -- (thirdss);
        \draw[->, blue] (thirds) -- (thirdsc);
        \draw[->, red] (thirds) -- (thirdssl);
        \draw[->, red] (thirds) -- (thirdssr);
    \end{tikzpicture}}
    \caption{A Game Tree for the Bergman Game of Depth Three.}
    \label{fig:bergman-game-tree}
\end{figure}
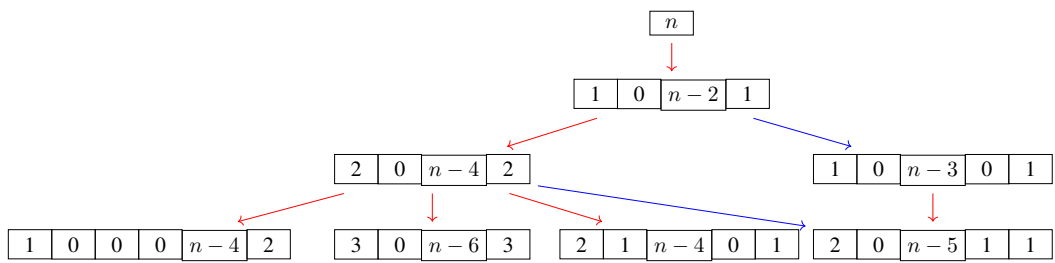

The behavior of games wherein each player makes one of all of their available moves with uniform probability at each term is currently completely unexplored. Through numerical analysis, we have been able to make conjectures concerning their behavior. However, there are no concretely proven results. This closely reflects the state of study for random Zeckendorf games (see \cite{baird-smithZeckendorfGame2020}). Our primary conjecture concerning game length says that as the number of summands grows the distribution of the number of moves used in a random game approaches a Gaussian (see \Cref{conj:gaussian-random-game}). We expect that the current techniques for proving things about these games will not easily apply to random games. The numerics for game length are explored in detail in \Cref{sec:numerics-random}. However, one possible direction is to examine the expected change in the important quantities $\chips{T}$, $\ind(T)$, and $\gap(T)$ throughout a given random game.

One could also define a Generalized Bergman Game on any \PLRS. The main difficulty in doing so arises when trying to generalize the split moves. Intuitively, the split move consists of a combine ``in reverse'' and then a combine forwards, as shown below for the Bergman Game:
\[
    (0, 0, 2, 0) \xrightarrow{\text{\quad Reverse Combine\quad}} (1,1,1, 0) \xrightarrow{\text{\quad Combine\quad}} (1, 0, 0, 1).
\]
In the case of \NIPLRS, the combine step here produces the $d_{p, j} = c_j - c_{p + j}$ entries found throughout this paper. However for other \PLRSs, these entries might be negative, which violates the rules of the game. 

\notdone{Shorten Future Work}

\appendix

\section{Numerics of Random Games}\label{sec:numerics-random}

Similarly to the numerics collected on the Zeckendorf Game in \cite{baird-smithZeckendorfGame2020}, we can consider the length of a random game played from an initial starting state ${}_0(n)$. For ease of coding, we only explored such numerics for the Bergman Game. However, we expect that similar numerical analysis of Generalized Bergman Games would yield similar results.

\begin{rmk}
    We choose to take the initial starting state to be ${}_0(n)$ so that we have a natural one-dimensional parameter to vary, as the limiting behavior if of the most interest here.
\end{rmk}

For clarity, we define what we mean by a random game.
\begin{defn}\label{defn:random-game}
    To play random game on an initial state $S$, list every available move from this initial state, and play one with a uniform probability. Continue this process from each new state created to generate a tree with probabilities attached to each edge. The probability of achieving a certain game (i.e., a path down the tree) is the product of all the probabilities associated to each edge.
\end{defn}

We then form the following conjecture based on the data collected from our C++ code, which can be found at 
\begin{center}
    \url{https://github.com/FayeAlephNil/public-bergman-code}
\end{center}
For any questions about the code, please email Faye Jackson (\href{mailto:alephnil@umich.edu}{alephnil@umich.edu}).

\begin{conjecture}\label{conj:gaussian-random-game}
    The distribution $\mathcal{L}_n$ of the length of a random game payed from the initial state ${}_0(n)$ converges to a Gaussian with linear mean $\mu_n \approx 2n + O(1)$. 
\end{conjecture}

This conjecture is supported by \Cref{fig:gaussian-random-game-1000,fig:gaussian-random-game-2000,fig:linear-mean-random-game-1000}. 

\begin{figure}[ht]
    \centering
    \includegraphics[scale=1.25]{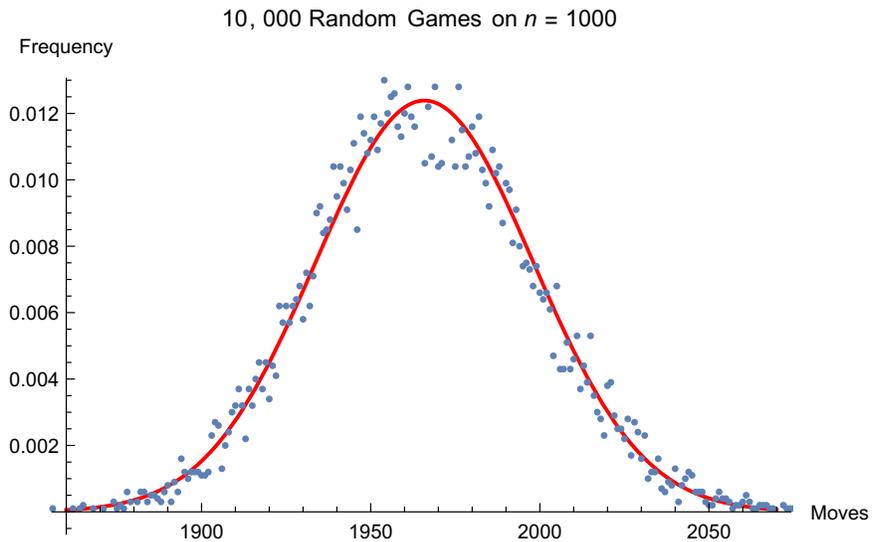}
    \caption{Frequency Graph of the Number of Moves in 10,000 simulations of the Bergman Game with random moves when $n = 1000$ with the best fit Gaussian overlaid in red.}
    \label{fig:gaussian-random-game-1000}
\end{figure}

\begin{figure}
    \centering
    \includegraphics[scale=1.25]{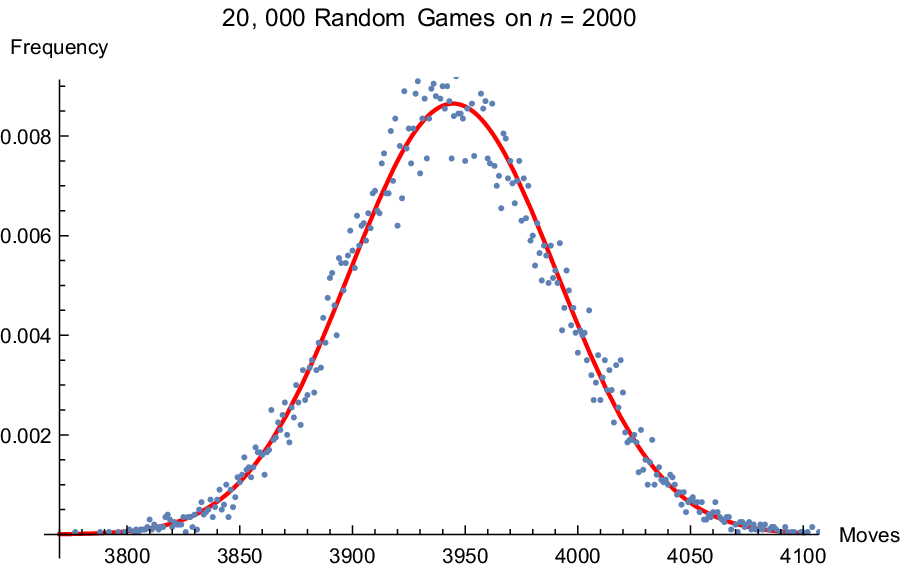}
    \caption{Frequency Graph of the Number of Moves in 20,000 simulations of the Bergman Game with random moves when $n = 2000$ with the best fit Gaussian overlaid in red.}
    \label{fig:gaussian-random-game-2000}
\end{figure}

\begin{figure}
    \centering
    \includegraphics[scale=1.25]{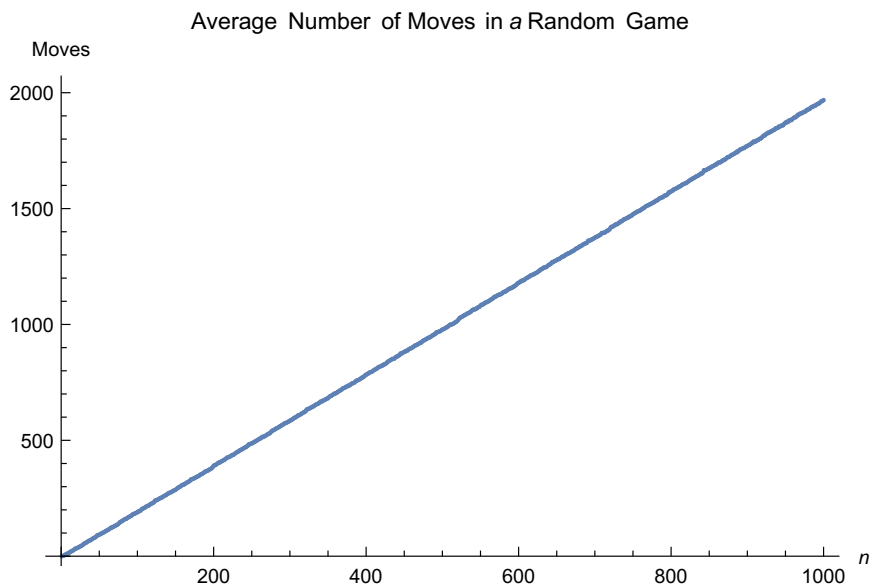}
    \caption{Graph of the average number of moves in random Bergman Games with initial state $n \in [1, 1000]$, averaging over 1,000 trials.}
    \label{fig:linear-mean-random-game-1000}
\end{figure}
The conjecture is also supported by the moments of the statistical sample from random games. After renormalizing to have mean zero and variance one the 20,000 random trials on $n = 2000$ give the following table of moments:
\begin{center}
\begin{tabular}{c||c|c|c|c|c|c}
     Moments & 3 & 4 & 5 & 6 & 7 & 8  \\
     \hline
     Data & 0.08 & 3.03 & 0.90 & 15.74 & 12.19 & 120.50 \\
     Gaussian & 0 & 3 & 0 & 15 & 0 & 105
\end{tabular}
\end{center}
Given the close agreement with the moments of the Gaussian, we have confidence that the first statement of the conjecture holds. For the second statement in the conjecture, numerics show that the best fit line to the average number of moves in a random game has linear coefficient $1.98$ with coefficient of determination $0.999991$, indicating a near perfect fit. In fact, we neglected to plot the best fit line in \Cref{fig:linear-mean-random-game-1000} because it was too difficult to see at the same time as the data.

Despite the strong numerical evidence, we expect that this conjecture is rather difficult to prove. Intuitively the length of a random game matches a Gaussian because so many independent random choices are performed as we traverse the tree. However, there is no clear recursive structure (i.e. a martingale interpreation) nor a way to express the process as a sum of binomial random variables in order to exploit a Central Limit Theorem and thereby prove the result. This puts the problem beyond the reach of our current techniques.

\section{Approximating Zero in base \texorpdfstring{$\psi = -1/\varphi$}{psi = -1/phi}}\label{sec:mn}

A natural question we initially ran into when trying to prove a left bound on the Bergman Game was how well one could approximate zero with finite power series $\sum_{i \geq 0} x_i\psi^i$ whose coefficients are all non-negative and $x_0 \geq 1$. This is formalized by the following definition.

\begin{defn}\label{defn:mn}
    For $n\geq 0$, let $\mathcal{N}(n)$ denote the set of all game states $S$ with $ \chips{S} \leq n$ and $S(i) = 0$ for $i<0$ (that is states which are nonzero only in non-negative indices). We then define
    \[
        m_n \coloneqq \inf_{S\in \mathcal{N}(n)} \abs{1 +  \valt(S)}.
    \]
    \notdone{BEN: Maybe this is secretly about approximating $-1$ in base $\psi?$}
\end{defn}

\begin{prop}\label{prop:mn-value}
    We have $m_n = \varphi^{-2n} = \psi^{2n}$
\end{prop}

Because we were able to calculate this explicitly, we include the calculation here, as well as how it allowed us to get a left bound on our games. However, the reader should note that this left bound is generalized by \Cref{cor:niplrs-left-bound}, and the proof provided there is more elementary when taken together with \Cref{prop:phi-final-left-bound}. Thus, this section survives only as an interesting calculable quantity. Notably, we have not been able to generalize this calculation to other \NIPLRS--even those whose coefficients all agree.

Assuming \Cref{prop:mn-value}, we may calculate a left bound.
\begin{cor}\label{cor:mn-left-bound}
    Let $S$ be some game state with leftmost summand at index zero and let $T$ be a game state occuring after $S$ in the Bergman Game. Then the leftmost summand of $T$ is at an index greater than or equal to $-2\chips{S} - \log_\varphi \chips{S} + 2$.
\end{cor}

\begin{proof}
    Let the leftmost summand of $T$ be located at index $a$, then we may write the following
    \begin{align*}
        \chips{S} &\geq \abs{\valt(S)} = \abs{\valt(T)} = \abs{\psi}^{a}\cdot\abs{1 + \valt({}_0T - {}_0(1))} \\
                  &\geq \abs{\psi}^{a} \cdot m_{\chips{T} - 1} \geq \abs{\psi}^{a}m_{\chips{S} - 1} = \abs{\psi}^{a + 2\chips{S} - 2}.
    \end{align*}
    Taking the $\log_\varphi$ on both sides yields $\log_\varphi \chips{S} \geq -a - 2\chips{S} + 2$, and rearrangement gives the desired bound.
\end{proof}

First we must establish that each $m_n > 0$ before we calculate the value.

\begin{lemma}\label{lemma:valt-nonneg}
    If $ \chips{S}>0$, then $ \valt(S) > 0$.
\end{lemma}

\begin{proof}
    Take a state ${}_aS_b$ with $ \chips{S} > 0$. We see that:
    \begin{align*}
        \valt(S) = \sum_{j = a}^b S(j)\psi^j = \psi^{a} \cdot \sum_{j = 0}^{b - a} S(j + a)\psi^j.
    \end{align*}
    Since $\psi^{a}$ is nonzero and all these terms are non-negative it suffices to show that $\sum_{j = 0}^{b - a} S(j + a)\psi^j$ is nonzero. Call $r = b - a$ and $y_j = S(j + a)$ for convenience.

    Note that $\Z[\varphi] \cong \Z[\psi]$ because $\varphi$ and $\psi$ are both roots of the irreducible polynomial $x^2 - x - 1$. Furthermore the function $f : \psi \mapsto \varphi$ uniquely defines a ring isomorphism $\Z[\psi] \to \Z[\varphi]$. Now we observe that:
    \begin{align*}
        f\left(\sum_{j = 0}^r y_j\psi^j\right) &= \sum_{j = 0}^r y_j\varphi^j \geq  \chips{S} > 0.
    \end{align*}
    Thus since $f$ is an isomorphism of rings, $\sum_{j = 0}^r y_j\psi^j \neq 0$. Therefore--by the reductions above--we win and obtain that $\valt(S) \neq 0$.

    It follows $ \valt(S) \neq - \valt(T)$ for any configurations $S, T$ with $ \chips{S}>0$ or $\chips{T} > 0$, as otherwise $ \valt(S+T) =  \valt(S) +  \valt(T)=0$ whereas $\chips{S + T} =  \chips{S} + \chips{T} > 0$.
\end{proof}

\begin{prop}\label{prop:mn-nonzero}
For every $n$ the constant $m_n$ is greater than zero.
\end{prop}

\begin{proof}
    We will prove this result by induction. Note that we trivially have that $m_0 = 1 > 0$, and likewise $m_1 = 1 -\abs{\psi} > 0$. For the inductive hypothesis, let $n \geq 2$ and $m_{n-1} > 0$.

    Now let $J>0$ such that $\abs{\psi^j} < \frac{m_{n-1}}{2}$ for all $j\geq J$, and let ${}_aS_b \in \mathcal{N}(n)$ with $S(a), S(b)>0$. There are finitely many instances $S$ with $b<J$. Why? Well, $ \chips{S} \leq n$, and so we must distribute at most $n$ summands among at most $J$ places, in the worst case where $a = 0$ and $b = K - 1$. 

    Let $\alpha$ denote the minimum $\abs{1 +  \valt(S)}$ over such $S$.  Note that 
    \[
        \abs{1 +  \valt(S)}= \abs{\valt({}_0(1)) +  \valt(S)} = \abs{\valt({}_0(1) + S)}.
    \]
    Thus by \Cref{lemma:valt-nonneg}, $1+ \valt(S)\neq 0$ and so $ \valt(S)\neq-1$ and $\alpha>0$. 
    
    If $b\geq J$ for some ${}_aS_b$, then the instance $S' = (S(a), \dots, S(b)-1)$ satisfies $\abs{1 +  \valt(S')} \geq m_{n-1} > 0$ by the induction hypothesis. By the triangle inequality
    \begin{align*}
        \abs{1 +  \valt(S)} &\geq \min(\alpha, \abs{1 +  \valt(S') + \psi^b}) \\
        &\geq \min(\alpha, \abs{1 +  \valt(S')} - \psi^b|) \\
        &\geq \min\left(\alpha, \frac{m_{n-1}}{2}\right) > 0.
    \end{align*}
    This finishes the inductive step.
\end{proof}

We may now proceed towards the explicit calculation.

\begin{lemma}\label{lemma:game-state-reduction}
    In our definition of $m_n$ we take an infimum over all non-negative game states $S$ with $ \chips{S} \leq n$. Instead, we show that we can take an infimum over reduced game states $S$, which only involve entries which are zero or one and also do not contain adjacent ones. Call this set $\mathcal{N}_R(n)$.
    
    Formally we show that $m_n = \inf_{S \in \mathcal{N}_R(n)} \abs{1 + \valt(S)}$
\end{lemma}

\begin{proof}
    Since $\mathcal{N}_R(n) \subseteq \mathcal{N}(n)$ it is trivial that:
    \begin{align*}
        \inf_{S\in \mathcal{N}(n)} \abs{1 +  \valt(S)} \leq  \inf_{T_r\in \mathcal{N}_R(n)} \abs{1 +  \valt(T_r)}
    \end{align*}
    We wish to derive the other inequality. To do so, it suffices to show that for any $S \in \mathcal{N}(n)$ there exists some state $T_r \in \mathcal{N}_R(n)$ such that $\abs{1 +  \valt(T_r)} \leq \abs{1 +  \valt(S)}$.

    We start with a process of ``quasi-reduction.'' Begin with some non-negative game state $S$ with $ \chips{S} \leq n$. Then by \Cref{prop:niplrs-term}, we may play the Bergman game on $S$ in order to get a reduced configuration $S_f$. By the invariance of $\valt$ we have that $\abs{1 + \valt(S)} = \abs{1 + \valt(S_f)}$.
    
    However, $S_f$ may contain nonzero entries in negative indices. We wish to get around this.  Write $S_f = (S_f(a), \ldots, S_f(0), \ldots, S_f(b))$. If $a \geq 0$ we're done and $S_f$ is a reduced game state in the non-negative indices with $\abs{1 + \valt(S_f)} = \abs{1 + \valt(S)}$.
    
    Suppose $a < 0$. Then we may ``pull the one'' into our game state by writing 
    \[
        T^0 := {}_0(S_f)_{b - a} - {}_0(1) + {}_{-a}(1).
    \]
    If $T^0$ is reduced, then we are done, as $T^0 \in \mathcal{N}_R(n)$ and 
    \begin{align*}
        \abs{1 + \valt(T^0)} &= \abs{1 - 1 + \psi^{-a} + \psi^{-a}\valt(S_f)} = \abs{\psi}^{-a}\abs{1 + \valt(S_f)} \\
                            &\leq \abs{\psi}\abs{1 + \valt(S)} < \abs{1 + \valt(S)}.
    \end{align*}
    However, $T^0$ might not be reduced. If $S_r(-a) = 1$ then $T^0$ will contain a two. In this case, we note that $T^0 \in \mathcal{N}(n)$ because $\chips{T^0} = \chips{S_f} - 1 + 1 \leq \chips{S}$. Therefore, we may iterate this process again, winning if $T^0_f$ has nonzero entries only in non-negative indices and otherwise constructing a $T^1 \in \mathcal{N}(n)$ with
    \[
        \abs{1 + \valt(T^1)} < \abs{\psi}\abs{1 + \valt(T^0)} < \abs{\psi}^2\abs{1 + \valt(S)}.
    \]
    If it terminates at any stage, i.e. if $T^j$ is reduced for some $j$ or $T^j_f \in \mathcal{N}(n)$ for some $j$ then we win by chasing the inequalities. If it does not terminate, then since $T^j \in \mathcal{N}(n)$ for all $j$, we will discover that we've created a contradiction to \Cref{prop:mn-nonzero}:
    \[
        \abs{1 + \valt(T^k)} < \abs{\psi}^k\abs{1 + \valt(S)}.
    \]
    But then we know that $m_n = 0$ because we may take $k \to \infty$ and $\abs{\psi} < 1$. Thus the process must terminate and we win!
\end{proof}

We are now prepared to move on to a proof of \Cref{prop:mn-value}. First, we write down a particularly important game state, which achieves the written bound

\begin{defn}\label{defn:prepped-to-combine}
    Define $P^n$ to be a game state with $P^n(2j - 1) = 1$ for all integers $1 \leq j \leq n$, and $P^n(j) = 0$ otherwise. Note that $\chips{P^n} = n$. We call such a state ``prepped to combine'' because the game state ${}_0(1) + P^n$ may be combined repeatedly to reach the game state ${}_{2n}(0)$. As demonstrated below for $n = 3$:
    \begin{center}
     \begin{tabular}{ccccccc}
         $\varphi^0$ & $\varphi^1$ & $\varphi^2$ & $\varphi^3$ & $\varphi^4$ & $\varphi^5$ & $\varphi^6$ \\
         \hline
         1 & 1 & 0 & 1 & 0 & 1 & 0 \\
         0 & 0 & 1 & 1 & 0 & 1 & 0 \\
         0 & 0 & 0 & 0 & 1 & 1 & 0 \\
         0 & 0 & 0 & 0 & 0 & 0 & 1 
     \end{tabular}
    \end{center}
\end{defn}

\begin{lemma}\label{lemma:prepped-to-combine-reduced}
    For any reduced game state $S$ whose summands are all in positive indices (i.e. $S(j) = 0$ for $j \leq 0$) we have $\valtabs(P^{\chips{S}}) \geq \valtabs(S)$.
    
    Furthermore, $\valtabs(P^{\chips{S}}) = 1 - \abs{\psi}^{2\chips{S}}$ by a simple calculation.
\end{lemma}

\begin{proof}
    For completeness we include the calculation of $\valtabs(P^{\chips{S}})$:
    \[
        \valtabs(P^{\chips{S}}) = \sum_{j = 1}^{\chips{S}} \abs{\psi}^{2j - 1} = \frac{\abs{\psi}(1 - \abs{\psi}^{2\chips{S}})}{1 - \abs{\psi}^2} = 1 - \abs{\psi}^{2\chips{S}}.
    \]
    We now move to showing the first statement.

    Notice that no two consecutive entries of $S$ are one and all entries of $S$ are less than or equal to one by definition. Then because $j \mapsto \abs{\psi}^j$ is a decreasing function, the most efficient arrangement is provided by placing a summand at index one, a summand at index three, and so on, just as specified by $P^{\chips{S}}$. Therefore $\valtabs(P^{\chips{S}}) \geq \valtabs(S)$.
\end{proof}

\begin{proof}[Proof of \Cref{prop:mn-value}]
    First, to show that $m_n\leq \varphi^{-2n}$, we calculate $\abs{1 + \valt(P^n)}$:
    \begin{align*}
        \abs{1 +  \valt(P^n)} = \abs{{}_0(1) + \valt(P^n)} = \abs{\valt({}_0(1) + P^n)} = \abs{\valt({}_{2n}(1))} = \varphi^{-2n}.
    \end{align*}
    Next, let $S \in \mathcal{N}(n)$. By \Cref{lemma:game-state-reduction}, we can assume that $S \in \mathcal{N}_R(n)$. To show that $\abs{1+ \valt(S)}\geq \varphi^{-2n}$, we will consider the function $\valtabs$. By the triangle inequality, $\abs{\valt(T)}\leq \valtabs(T)$ for all game states $T$.  
    
    We now divide into cases depending on if $S(0) = 1$. Repeatedly using \Cref{lemma:prepped-to-combine-reduced}.
    \begin{itemize}
        \item Suppose that $S(0) = 1$. Let $S'$ denote $S$ after removing the summand in the zeroth index. Then
            \begin{align*}
                \abs{1+ \valt(S)} &= \abs{2+ \valt(S')} \geq 2 - \valtabs(S') \\
                &\geq 2 - \valtabs(R_{0,\chips{S'}}) = 1 + \abs{\psi}^{2\chips{S'}} > \varphi^{-2n}.
            \end{align*}
            Thus $S$ performs worse than our explicit construction $P^n$ above.
        \item Suppose that $S(0) = 0$. Then we have that
            \[
                 \abs{1+ \valt(S)} \geq 1 - \valtabs(S) = 1 - \valtabs(P^n) = \abs{\psi}^{2\chips{S}} \geq \varphi^{-2n}.
            \]
            Thus $S$ performs either worse than or as well as $P^n$.
    \end{itemize}
    This then implies that $m_n\geq \varphi^{-2n}$. Together, these two pieces complete the proof.
\end{proof}

\section{Termination of a wide class of Mass-Conservative Games}\label{sec:term-mass-cons}

Thus far, we have examined games on the infinite tuple arising from recurrence relations and questions in number theory, and we have examined why and how quickly these games terminate.  We now turn to a natural generalization of this line of study: \textit{Which possible move-sets for some game on the tuple ensure that the game terminates?}  We begin with a couple of definitions. In this section, we refer to the entries in the tuple as ``chips'' or ``tokens'' rather than as a summand. We do so because there is no intrinsic connection to base $\beta$ decompositions within this general class of games, and so using the word ``summand'' gives the wrong connotation.  So as to properly and rigorously define our conditions for termination, we give more specific definitions for certain terms to be used in this appendix.

\begin{defn}
    Given any two states on the tuple $S,T$, let a move denote the immediate replacement of $S$ with $T$.
\end{defn}

\begin{defn}
    Let $m$ denote a certain set of moves in a game on the tuple.  We say $m$ is a locally defined move (LDM) if $m$ can be defined in the following way:  $m$ has an assigned `initial state' $m_i$, a finite string on the tuple.  Let $s_i$ be the set of indices at which $m_i$ is non-zero.  Further, $m$ has a `final state' $m_f$ which is another finite string on the tuple.  Akin to before, let $s_f$ be the set of indices at which $m_f$ is non-zero.  The move defined by $S \rightarrow S-m_i+m_f$ is an LDM.  Note that if no subset of the summands on the tuple before the move takes place constitute $m_i$, then the move cannot be performed as we do not allow negative numbers of chips.
\end{defn}

\begin{defn}
    Given an LDM $m$, let $m^k$ refer to the move derived from $m$ by shifting $m_i$ and $m_f$ both $k$ indices to the right. Specifically, $m=m^0$.  We may then define the translation class of $m$ as $M=\{m^k:k\in\Z\}$.
\end{defn}

\begin{example}
    We may create a move type with $m_i = _0(2)$ and $m_f = {}_{-2}(1,0,0,1)_1$.  The following represents an instance of this move being performed:
    \begin{align*}
        _{-4}(3,5,0,0,3,0,0,2)_3 \rightarrow _{-4}(3,5,1,0,1,1,0,2)_3.
    \end{align*}
    Notice that the move may not be performed from a state like $_{-1}(1,1,2,1)_3$ because no subset of the chips form $m_i = _0(2)$.  Also recall that this example move is in fact the split move at the 0-index from the AZG.  The translation class $M$ of $m$ is then the split move more generally, performed at any index.
\end{example}

\begin{rmk}
    In the specific games studied elsewhere in our paper, we have thought of translation classes of LDMs simply as moves themselves (like the split or combine moves), and so when the context is clear, we will continue to use the singular "move" to discuss a translation class.
\end{rmk}

\begin{defn}
    Define a generalized left split move (GLSM) to be the the translation class $M$ of an LDM $m$ with initial state $_a(m_i)_b$ and final state $_c(m_f)_d$ such that $c<a$ and $d\geq b$.  Similarly, a generalized right split move (GRSM) satisfies $c\leq a$ and $d>b$.  If $c<a$ and $d>b$, we refer to $M$ simply as a generalized split move (GSM).
\end{defn}

\begin{defn}
    We say an LDM $m$ (or its translation class $M$) is mass-increasing if $\chips{m_i}<\chips{m_i}$, mass decreasing if $\chips{m_i}>\chips{m_i}$, and as mass conservative if $\chips{m_i}=\chips{m_i}$.  We define mass non-increasing and non-decreasing moves in the usual way.
\end{defn}

\begin{prop}\label{tuple-term}
    A mass non-increasing game on the tuple (given by a set of rules) terminates iff for all possible initial states, there is a finite interval of indices which contains all possible subsequent game states, and there are no two sequences of moves which may have the effect of changing game states $S\rightarrow T$ and $T\rightarrow S$, respectively, for any game states $S, T$.  Equivalently, we may phrase the second condition as there being no way to achieve a game state twice during a game. 
\end{prop}

\begin{proof}
    If the game does not take place in a finite interval (and thus fails the first condition), then it has either unbounded leftmost or rightmost index.  Given that any move can only increase the leftmost or rightmost index by a finite amount, there must be an achievable infinite game.  Alternatively, if there is a game state which may be achieved twice during a game, then it may be achieved infinite times in an endless loop, so there is an infinite game.
    
    If, however, both conditions are met, then the game terminates.  Given any initial state, the stipulation that the resulting game takes place within a finite number of indices implies that there are a finite number of achievable game states - namely, all the ways of placing the initial chips inside the bounding interval.  Then if no state may be achieved twice, the number of these states bounds the length of the game, which thus terminates.
\end{proof}

Such a condition absolutely determines termination, but it may often prove difficult to use in practice.  That is, given a list of legal moves, does the game they define terminate?  In principal, one may have to check an infinite number of move strings to ensure that no state may be reached twice given any initial condition.  As such, we give now a stronger but more easily verifiable condition for game termination.

\begin{prop} \label{tuple-term2}
    Let $\mathscr{M}$ be a finite set of LDMs and translation classes of LDMs.  Suppose that all $M\in\mathscr{M}$ are mass non-increasing and either every mass-conservative move is a GRSM or they are all GLSMs.  Then any game $G(\mathscr{M})$ played with all moves from $\mathscr{M}$ terminates.
\end{prop}

\begin{proof}
    Since our game is mass non-increasing, we may only decrease the number of chips a finite number of times.  Thus, to prove termination, we need only prove that one may only conduct a finite number of consecutive mass conservative moves.  Thus, we assume from now on that all $M\in\mathscr{M}$ are mass conservative WLOG.
    
    For much the same reason as splits in the Generalized Bergman Game, GRSMs and GLSMs cannot increase the max gap size $\gap(S)$ of a state $S$ (see \Cref{lemma:niplrs-gap-change}).  By construction, there is a well-defined max gap size in the final state of any $M\in\mathscr{M}$.  Call this quantity $\gap(\mathscr{M})$.  Thus, the max gap size over the course of a game $G(\mathscr{M})$ with initial state $S_i$ cannot exceed $\gap(G) := \max(\gap(\mathscr{M}),\gap(S_i))$, the maximal max gap size over the entire game.  Then the width of any game state during the game cannot exceed $\chips{I_i} + (\chips{I_i}-1)g(G)$.
    
    Notice that none of our moves increase the left bound or decrease the right bound of any game state.  Then any left bound during the game must be left of the initial left bound, and similarly for the right.  Combining this fact with the finite width of any achieved game state, we have that the entire game $G(\mathscr{M})$ must take place entirely on a finite window of $w+1$ adjacent indices.
    
    Suppose our moves are all GRMSs.  We now re-index our game so that the leftmost index in the window of play is 0 and the rightmost is $w$.  Note that we may have to index some of the initial and final states of our moves as well if they are not translation invariant.  Let $\chips{\mathscr{M}}$ denote the largest number of chips used in any move $M\in\mathscr{M}$ (i.e., the largest number of chips in any initial or final state).  We define the following monovariant on a game-state $S$:
    \begin{align*}
        f(S) \coloneqq \sum_{i=0}^w S(i)(\chips{\mathscr{M}})^i
    \end{align*}
    Observe that weights increase to the right.  Since our moves are GRSMs, the legal move that would decrease this quantity the most is the one which begine with $\chips{\mathscr{M}}$ chips on index $w-1$ and ends with one chip on index $w$ and the remaining $\chips{\mathscr{M}}-1$ chips at index 0.  However, even in this case, $f(S)$ increases by
    \begin{align*}
        \chips{\mathscr{M}}-1 + (\chips{\mathscr{M}})^w - \chips{\mathscr{M}}(\chips{\mathscr{M}})^{w-1} = (\chips{\mathscr{M}})^w - 1 >0,
    \end{align*}
    assuming any nontrivial moves.  Thus $f(S)$ always increases when a move is made, and since it only takes on integer values, it must increase by at least one.  Further, it must lie between $n$ and $n (\chips{\mathscr{M}})^w$, where $n$ is the number of chips in play.  Therefore, only $n((\chips{\mathscr{M}})^w-1)$ moves can be made, and so the game terminates.  The argument for the case of GLSMs follows similarly, with an analogous monovariant $\sum_{i=0}^w S(i)(\chips{\mathscr{M}})^{w-i}$ which increases to the left.
\end{proof}

\begin{rmk}
    Interestingly, we may also generalize these results to games on the infinite $d$-dimensional lattice, for any finite $d$.  Instead of requiring moves to be GRSM or GLSM (which is a one-dimensional concept), we instead look at the projection of each move onto each axis, i.e., how each move changes any particular component of the chip distribution.  When we require all such projections of each move onto each axis to be GRSMs or GLSMs, termination follows.
\end{rmk}

\section{Explicit Bounds on Game Length}

Within this section we derive more explicit bounds than those given in \Cref{prop:niplrs-game-length-nonabsolute} and \Cref{thm:niplrs-game-length}. In particular, the first proposition now contains no asymptotics, and in the second we calculate the coefficient on the dominating term of $\chips{S}^2$. The ideas of each proof directly mirror those provided in \Cref{sec:niplrs-n2-term}, and so these serve merely as more careful and explicit bounds. 

\begin{prop}\label{prop:niplrs-game-length-nonabsolute-explicit}
    Let ${}_aS_b$ be some initial game state. Then we know that the Generalized Bergman Game for a \NIPLRS of depth at least two with coefficients $c_1 \geq \cdots \geq c_k \geq 1$ played on this initial game state terminates in at most
    \begin{align*}
        &\frac{k}{2(C_0 - 1)^2}\chips{S}^2 + \frac{2}{C_0 - 1}\chips{S}\log_\rho \chips{S} \\
        &+ \left(\frac{1}{C_0 - 1} + \frac{2A + (b - a)(R + 1)}{C_0 - 1} + \frac{2C_0R(b - a) + 2C_1 -  k}{2(C_0 - 1)^2}\right)\chips{S}.
    \end{align*}
    moves. For the definitions of these constants refer to \Cref{fig:table-constants}.
\end{prop}

\begin{proof}
    We use the same proof strategy as \Cref{prop:niplrs-game-length-nonabsolute}, and we take time to explicitly bound the following sum, which is the maximum increase in $\ind(T)$ over the course of the entire game:
    \[
        \sum_{r = 1}^{\frac{\chips{S}}{C_0 - 1}} \left[-k\frac{\chips{S}}{C_0 - 1} - \log_\rho \val(S) -A + rk - \sum_{j = 1}^k c_j\left(-k\frac{\chips{S}}{C_0 - 1} - \log_\rho \val(S) -A + rk - j\right)\right].
    \]
    To do so we bound $\log_\rho \val(S)$ above by $Rb + \log_\rho \chips{S}$ and below by $\log_\rho \chips{S}$ to see that $\ind(T)$ increases by at most
    \begin{align*}
        \sum_{r = 1}^{\frac{\chips{S}}{C_0 - 1}} \left[-k\frac{\chips{S}}{C_0 - 1} - \log_\rho \chips{S} -A + rk - \sum_{j = 1}^k c_j\left(-k\frac{\chips{S}}{C_0 - 1} - Rb - \log_\rho \chips{S} -A + rk - j\right)\right] .
    \end{align*}
    Evaluating this sum gives an increase of
    \begin{align*}
        \Delta^+\ind :&= \frac{k\chips{S}^2}{C_0 - 1}+ \frac{C_0Rb}{C_0 - 1}\chips{S} + \chips{S}\log_\rho \chips{S} + A \chips{S}- \frac{(C_0 - 1)k}{2}\frac{\chips{S}}{C_0 - 1}\left(\frac{\chips{S}}{C_0 - 1} + 1\right)  \\
        &= \frac{k}{2(C_0 - 1)}\chips{S}^2 + \chips{S}\log_\rho\chips{S} + \left(A + \frac{2C_0Rb + 2C_1 - k}{2(C_0 - 1)}\right)\chips{S} 
    \end{align*}
    in $\ind(T)$ over the course of the whole game. In order to not break the lower bound on $\ind(S_f)$ we have the following maximum total number of splits across the course of the entire game:
    \[
        \frac{1}{C_0 - 1} \cdot \left[\Delta^+\ind + Rb\chips{S} + \chips{S}\log_\rho \chips{S} + A\chips{S} + b\chips{S}\right].
    \]
    Then because there are at most $\frac{\chips{S}}{C_0 - 1}$ combines throughout the course of the game, we can simply add together these two contributions to find the upper bound on the total number of moves in the game indicated in the proposition.
\end{proof}

\begin{prop}\label{prop:niplrs-game-length-explicit}
    Given any initial state $S$, the Generalized Bergman Game terminates in at most
    \[
        \left[\frac{(R + 1)(C_0 - 1) + C_0R}{(C_0 - 1)^2}\left(\frac{2k}{C_0 - 1} + k + 1\right) + \frac{k}{2(C_0 - 1)^2}\right]\chips{S}^2 + O(\chips{S}\log \chips{S})
    \]
    moves. For the definitions of these constants refer to \Cref{fig:table-constants}.
\end{prop}

\begin{proof}
    Using \Cref{prop:niplrs-game-length-nonabsolute-explicit} along with the proof strategy of \Cref{thm:niplrs-game-length} we see that the number of moves is bounded above by the following for some constants $X, Y$ given in the statement of the proposition in terms of the coefficients $c_1, \ldots, c_k$:
    \begin{align*}
        \sum_i \left[\frac{k}{2(C_0 - 1)^2}\chips{S^i}^2 + X\chips{S^i}\log_\rho \chips{S^i} + Y\chips{S^i} + \frac{(R + 1)(C_0 - 1) + C_0R}{(C_0 - 1)^2}\len(S^i)\chips{S^i}\right].
    \end{align*}
    Using the fact that $\sum_i \chips{S^i} = \chips{S}$ we may then bound each piece of this easily:
    \begin{align*}
        \sum_i \chips{S^i}^2 &\leq \left(\sum_i \chips{S^i}\right)^2 = \chips{S}^2 \\ \\
        \sum_i \chips{S^i}\ceil{\log_\rho \chips{S^i}} &\leq \log_\rho \chips{S} \cdot \sum_i \chips{S^i} = \chips{S}\log_\rho \chips{S} \\ \\
        \sum_i \len(S^i)\chips{S^i} &\leq \chips{S}\sum_i \len(S^i) \leq \left(\frac{2k}{C_0 - 1} + k + 1\right)\chips{S}^2
    \end{align*}
    In total, this provides a bound on the length of a game $G$ with initial state $S$ by 
    \[
        \left[\frac{(R + 1)(C_0 - 1) + C_0R}{(C_0 - 1)^2}\left(\frac{2k}{C_0 - 1} + k + 1\right) + \frac{k}{2(C_0 - 1)^2}\right]\chips{S}^2 + O(\chips{S}\log \chips{S})
    \]
    moves, completing the proof.
\end{proof}

\printbibliography

\end{document}